\documentclass[a4paper, 11pt, twoside, notitlepage]{amsart}

\usepackage{amsmath,amscd}
\usepackage{amssymb}
\usepackage{amsthm}
\usepackage{comment}
\usepackage{graphicx, xcolor}

\usepackage{mathrsfs}
\usepackage[ocgcolorlinks, linkcolor=blue]{hyperref}
\usepackage[normalem]{ulem}

\newcommand\redsout{\bgroup\markoverwith{\textcolor{red}{\rule[0.5ex]{2pt}{0.4pt}}}\ULon}

\usepackage{bm}
\usepackage{bbm}
\usepackage{url}

\usepackage[utf8]{inputenc}
\usepackage{mathtools,amssymb}
\usepackage{esint}
\usepackage{tikz}
\usepackage{dsfont}
\usepackage{relsize}
\usepackage{url}
\usepackage{xcolor}
\usepackage{graphicx}
\usepackage{mathrsfs}
\usepackage[shortlabels]{enumitem}
\usepackage{lineno}
\usepackage{amsmath}
\usepackage{enumitem}
\usepackage{amsthm} 
\usepackage{verbatim}
\usepackage{dsfont}
\numberwithin{equation}{section}

\newcommand{\op}{\overline \p}

\mathtoolsset{showonlyrefs}

\DeclareMathOperator{\tr}{tr}
\DeclareMathOperator{\Cof}{Cof}

\allowdisplaybreaks
\newcommand{\para}[1]{\vspace{3mm} \noindent\textbf{#1.}} 

\graphicspath{{images/}}

\newtheorem{theorem}{Theorem}[section]
\newtheorem{lemma}[theorem]{Lemma}
\newtheorem*{lemma*}{Lemma}

\newtheorem{proposition}[theorem]{Proposition}

\newtheorem{remark}[theorem]{Remark}

\title[Inverse source problem for the Monge-Amp\`ere equation]{An inverse problem for the Monge-Amp\`ere  equation}

\author[T. Liimatainen]{Tony Liimatainen}
\address{Department of Mathematics and Statistics, University of Helsinki, Helsinki, Finland}
\curraddr{}
\email{tony.liimatainen@helsinki.fi}

\author[Y.-H. Lin]{Yi-Hsuan Lin$^{*}$}
\thanks{$^{*}$Corresponding author. Email: \texttt{yihsuanlin3@gmail.com}.}
\address{Department of Applied Mathematics, National Yang Ming Chiao Tung University, Hsinchu, Taiwan \& Fakult\"at f\"ur Mathematik, University of Duisburg-Essen, Essen, Germany}
\curraddr{}
\email{yihsuanlin3@gmail.com}

\keywords{Inverse source problem, Monge-Amp\`ere equation, fully nonlinear, anisotropic Calder\'on problem,  stationary phase method}
\subjclass[2020]{35R30, 35J25, 35J60, 35J96}

\newcommand{\C}{{\mathbb C}}
\newcommand{\R}{{\mathbb R}}

\newcommand{\N}{{\mathbb N}}

\newcommand{\eps}{\epsilon}

\newcommand {\p} {\partial}

\newcommand{\LC}{\left(}
\newcommand{\RC}{\right)}
\newcommand{\wt}{\widetilde}

\newcommand{\abs}[1]{\left\lvert #1 \right\rvert}

\begin{document}

\begin{abstract}
	We prove uniqueness for an inverse source problem for the Monge--Amp\`ere equation $\det D^2u=F$ on a bounded uniformly convex planar domain. The source is smooth and strictly positive, and its boundary derivatives through order two are prescribed. We show that the Dirichlet-to-Neumann map on the convex solution branch, restricted to a neighborhood of zero boundary data, determines the source throughout the domain. The first linearization gives an anisotropic equation whose principal coefficient is the inverse Hessian of an unknown background solution. We determine its conformal class in the original Euclidean coordinates and use the second linearized equation to determine the remaining scalar factor. We construct complex geometric optics solutions and prove arbitrary order expansions of their correction terms in $W^{2,p}$ for phases without critical points. A termwise stationary phase analysis of the second linearized identity gives a differential equation with a Cauchy integral term. We prove the required unique continuation property by incorporating that integral as an additional unknown in a first order system and applying a Carleman estimate.
\end{abstract}
	\maketitle

	\tableofcontents

\section{Introduction}\label{sec: introduction}
The Monge--Amp\`ere equation is a fully nonlinear equation whose classical origins include \cite{monge1784memoire,ampere1819memoire}. Its general form is
\begin{equation}\label{general MA equation}
	\det D^2u=f(x,u,\nabla u)\quad\text{in }\Omega,
\end{equation}
where $\Omega\subset\R^n$, $D^2u$ is the Hessian matrix, and $f$ is prescribed. On the branch of solutions with positive definite Hessian, its linearization is elliptic. The regularity and solvability of the Dirichlet problem depend on the assumptions on the domain, the right-hand side and the boundary values; see \cite{Figalli17} and Remark~\ref{remark: source and domain conditions}.

The equation occurs in several geometric problems.
\begin{itemize}
	\item For optimal transport with quadratic cost, a sufficiently regular transport map $\nabla u$ between positive densities $\rho_0$ and $\rho_1$ satisfies $\det D^2u(x)=\rho_0(x)/\rho_1(\nabla u(x))$. We refer to \cite{Villani_09_optimal} for this equation and applications of optimal transport.
	\item The graph of a function satisfying $\det D^2u=K(x)(1+|\nabla u|^2)^{(n+2)/2}$ has prescribed Gaussian curvature $K$. Prescribing curvature of convex hypersurfaces also leads to the classical Minkowski problem \cite{minkowski1897allgemeine,minkowski1903}, formulated in terms of the Gauss map. These are related geometric settings, rather than the same boundary value problem.
	\item In complex geometry, the existence of Ricci-flat K\"ahler metrics on compact K\"ahler manifolds with vanishing first Chern class is obtained by solving a complex Monge--Amp\`ere equation; see \cite{Yau_CPAM_1978,Aubin_1982_MA}.
\end{itemize}
Further geometric equations and regularity results are discussed in \cite{calabi_1972,CY_1986,TW00_Berstein,TW02_affine,TW05_Plateau,Figalli17}. These examples concern the forward Monge--Amp\`ere equation. The boundary measurements studied below are those of its classical Dirichlet problem.

We consider the spatially dependent source
\begin{equation}\label{spatial source}
	f(x,u,\nabla u)=F(x),
\end{equation}
so that \eqref{general MA equation} becomes
\begin{equation}\label{general MA equation 2}
	\det D^2u=F(x).
\end{equation}
We ask whether the Dirichlet-to-Neumann (DN) map determines $F$. The determinant is the Jacobian of the gradient map $\nabla u$. The Dirichlet value determines the tangential components of this map on the boundary, and the normal derivative supplies the remaining component. Thus, the inverse problem concerns the determination of an interior Jacobian density from these boundary measurements.

The first linearization already contains an unknown principal coefficient. If $u_0$ is the convex solution with zero Dirichlet data, that coefficient is $(D^2u_0)^{-1}$. The source enters the principal part through an unknown background solution. We first determine the conformal class of $g=D^2u_0$ in the original Euclidean coordinates. We then analyze the second linearized identity, including the first derivative terms introduced when Euclidean Hessians are expressed in isothermal coordinates. The analysis uses both local stationary phase contributions and a nonlocal Cauchy integral contribution.

Linearization has long been used in inverse boundary value problems for nonlinear equations. It gives linear equations and boundary measurements by differentiating the nonlinear DN map. Early applications include \cite{isakov1993uniqueness_parabolic}, and second order linearizations were used in \cite{AYT2017direct,CNV2019reconstruction,KN002,sun1996quasilinear,sun2010inverse,sun1997inverse}.

Interactions between solutions of linearized equations have been used for nonlinear hyperbolic equations in \cite{KLU2018} and semilinear elliptic equations in \cite{FO20,LLLS2019nonlinear}. Results for semilinear elliptic equations, including partial boundary measurements, appear in \cite{LLLS2019partial,LLST2022inverse,KU2019remark,KU2019partial,FLL2021inverse}. Quasilinear elliptic equations are studied in \cite{KKU2022partial,CFKKU2021calderon,LW24_quasi}, and inverse problems for minimal surface equations in \cite{ABN_20_minimal,carstea2024inverse,CLLT2023inverse_minimal,CLT24,nurminen2023inverse}. A connection between minimal surface inverse problems and AdS/CFT is considered in \cite{Jokela2025}. Related results for semilinear and fractional equations can be found in \cite{LL2019global,LL2020inverse,LSX_22_IP,harrach2022simultaneous,ST_23_single,LL25_book}. We refer to \cite{lassas2025introduction} for a survey.

For an inverse source problem, the solution with zero Dirichlet data generally does not vanish. Source recovery and the resulting gauge invariances for semilinear elliptic and reaction--diffusion equations were studied in \cite{LL24_elliptic_source,KLL_reaction_source}. In those problems, the background solution enters lower order coefficients of the linearized equation. Here it determines the principal coefficient itself. The second linearization also contains products of Hessians, rather than only products of solutions or their first derivatives. We give the CGO constructions and their correction estimates separately, including the arbitrary order expansion for a noncritical phase, before using them in this integral identity.

	\subsection{Mathematical formulations and main results}
	Let $\Omega\subset\R^2$ be a bounded uniformly convex domain with $C^\infty$ boundary, and let $F\in C^\infty(\overline\Omega)$ satisfy $F\geq c_0>0$. We consider
	\begin{align}\label{MA equation}
		\begin{cases}
			\det D^2u=F&\text{in }\Omega,\\
			u=\varphi&\text{on }\partial\Omega
		\end{cases}
	\end{align}
	on the convex solution branch. Positivity of the source does not by itself specify this branch.
	
	Let $u_0$ be the convex solution with zero Dirichlet data. Proposition~\ref{prop: well-posed} gives a smooth solution map near $u_0$ for boundary values sufficiently small in $C^{4,\alpha}(\partial\Omega)$. The solutions on this branch have uniformly positive definite Hessians. For $\delta>0$ sufficiently small, define
	\begin{equation}\label{DN map}
		\Lambda_F:B_\delta(\partial\Omega)\to C^\infty(\partial\Omega),\quad
		\Lambda_F(\varphi)=\partial_\nu u_\varphi|_{\partial\Omega},
	\end{equation}
	where $B_\delta(\partial\Omega)$ is defined in \eqref{admissible boundary data}, $u_\varphi$ is the solution on this branch, and $\nu$ is the Euclidean unit outer normal. In particular, $\partial_\nu u_\varphi=\nu\cdot\nabla u_\varphi$. We consider the following inverse problem.

	\begin{enumerate}[(\textbf{IP})]
		\item\label{IP} \textbf{Inverse Source Problem.} Can we determine the unknown source $F$ in $\Omega$ by using the knowledge of the DN map $\Lambda_F$?
		
	\end{enumerate}
	
\begin{remark}\label{remark: inverse source linear}
	The corresponding inverse source problem for a linear equation has an obstruction even when its full DN map is known. Consider
	\begin{equation}\label{equ: linear source 1}
		\begin{cases}
			\Delta u=F&\text{in }\Omega,\\
			u=\varphi&\text{on }\partial\Omega.
		\end{cases}
	\end{equation}
	Choose any $\psi\in C_c^\infty(\Omega)$ and set $v=u+\psi$. Then
	\begin{equation}\label{equ: linear source 2}
		\begin{cases}
			\Delta v=F+\Delta\psi&\text{in }\Omega,\\
			v=\varphi&\text{on }\partial\Omega,
		\end{cases}
	\end{equation}
	and $\partial_\nu v=\partial_\nu u$ on the boundary.
	\begin{enumerate}[\rm(i)]
		\item The same $\psi$ works for every Dirichlet value. Hence $F$ and $F+\Delta\psi$ have the same DN map. They also agree in a neighborhood of the boundary, so prescribing all boundary derivatives of the source does not remove this ambiguity.
		\item Nonlinearity can remove source ambiguities in suitable classes of equations. This was proved for semilinear elliptic and reaction--diffusion equations in \cite{LL24_elliptic_source,KLL_reaction_source}. Related results for quasilinear elliptic equations and semilinear wave equations appear in \cite{liimatainen25_mean_curvature,qiu2025uniqueness}. Theorem~\ref{theorem: uniqueness} gives source uniqueness for the convex branch of the planar Monge--Amp\`ere equation.
	\end{enumerate}
\end{remark}

	Fix $\alpha\in(0,1)$ and define
	\begin{equation}\label{admissible boundary data}
		B_\delta(\partial\Omega)=\{\varphi\in C^\infty(\partial\Omega):\|\varphi\|_{C^{4,\alpha}(\partial\Omega)}<\delta\}.
	\end{equation}
	The following theorem gives uniqueness from the DN map on this neighborhood of zero boundary data.

\begin{theorem}[Unique determination]\label{theorem: uniqueness}
	Let $\Omega\subset\R^2$ be a bounded uniformly convex domain with $C^\infty$ boundary. A smooth source $F\geq c_0>0$, whose boundary derivatives through order two are prescribed, is uniquely determined by the DN map \eqref{DN map} on a sufficiently small neighborhood of zero boundary data.
	
	More precisely, let $F_1,F_2\in C^\infty(\overline\Omega)$ satisfy $F_j\geq c_0>0$, and let $\Lambda_{F_j}$ be the DN map on the convex solution branch of
	\begin{equation}\label{MA equation j=1,2}
		\begin{cases}
			\det D^2u^{(j)}=F_j&\text{in }\Omega,\\
			u^{(j)}=\varphi&\text{on }\partial\Omega,
		\end{cases}
	\end{equation}
	for $j=1,2$. Assume that $\partial^\beta F_1|_{\partial\Omega}=\partial^\beta F_2|_{\partial\Omega}$ for all $|\beta|\leq2$. If
	\begin{equation}\label{eq: same DN map}
		\Lambda_{F_1}(\varphi)=\Lambda_{F_2}(\varphi)\quad\text{for every }\varphi\in B_\delta(\partial\Omega),
	\end{equation}
	where $\delta>0$ is sufficiently small that both local solution maps are defined, then $F_1=F_2$ in $\Omega$.
\end{theorem}
	
	We emphasize that the smallness condition on the boundary data is imposed to ensure that the forward problem is locally well-posed and that solutions depend smoothly on the boundary values. This smooth dependence allows us 
	to differentiate the equation with respect to boundary parameters and implement the higher order linearization method.

\begin{remark}\label{remark: source and domain conditions}
	The assumptions in Theorem~\ref{theorem: uniqueness} have the following roles.
	\begin{enumerate}[\rm(i)]
		\item\label{item 1 rmk uniquness} \textbf{Uniform convexity and positivity.} The Dirichlet theory gives a regular convex background solution with positive definite Hessian. Its uniform positivity on $\overline\Omega$ makes the linearized Dirichlet problem uniformly elliptic. Proposition~\ref{prop: well-posed} then gives a smooth local solution map on an open ball in the boundary H\"older space.
		\item\label{item 2 rmk uniquness} \textbf{Boundary derivatives of the source.} These derivatives, together with $\Lambda_F(0)$, determine the boundary derivatives of $u_0$ through order four; see Lemma~\ref{lemma: boundary determination}. They supply the metric jets used in the first linearized inverse problem and the vanishing boundary jets used to extend the coefficient in the second integral identity by zero.
	\end{enumerate}
	The smallness condition concerns the boundary perturbations. The source is not assumed to be small or close to a constant. The boundary information about the source is part of the theorem's hypotheses.
\end{remark}

	Let $u_0$ be the convex solution with zero Dirichlet data, and set $g=D^2u_0$. The first linearization gives
	\begin{equation}\label{first linearized equation in the intro}
		u_0^{ab}\partial_{ab}v=0\quad\text{in }\Omega,
	\end{equation}
	where $(u_0^{ab})=(D^2u_0)^{-1}$. Thus, $g$ is the metric, whereas $g^{-1}$ is the principal coefficient matrix. The first derivative of the nonlinear DN map gives the Euclidean normal derivatives of these solutions. Boundary determination supplies the metric information needed to convert them to the conormal data defined in \eqref{DN_linearized}. We use the following inverse result.

\begin{theorem}\label{theorem: J=Id in intro}
	Let $\Omega\subset\R^2$ be a bounded simply connected domain with smooth boundary. Suppose that $g_1,g_2$ are smooth positive definite symmetric matrix fields on $\overline\Omega$ and that $\partial^\beta g_1|_{\partial\Omega}=\partial^\beta g_2|_{\partial\Omega}$ for $|\beta|\leq1$. For a metric $g$, consider
	\begin{equation}\label{equation g_ab}
		g^{ab}\partial_{ab}v=0\quad\text{in }\Omega,
	\end{equation}
	where $(g^{ab})=g^{-1}$, and define its DN map using the conormal density in \eqref{DN_linearized}. In particular, for $j=1,2$, let $\Lambda'_{g_j}$ be the DN map of
	\begin{equation}\label{non-divergence elliptic equ in intro}
		\begin{cases}
			g_j^{ab}\partial_{ab}v_j=0&\text{in }\Omega,\\
			v_j=\phi&\text{on }\partial\Omega.
		\end{cases}
	\end{equation}
	If $\Lambda'_{g_1}=\Lambda'_{g_2}$, then there is a positive $c\in C^\infty(\overline\Omega)$, with $c|_{\partial\Omega}=1$, such that
	\begin{equation}\label{g_1 = c g_2 in intro}
		g_1=cg_2\quad\text{in }\Omega.
	\end{equation}
\end{theorem}

The equality in \eqref{g_1 = c g_2 in intro} holds in the original Euclidean coordinates. A nonlinear change of coordinates generally introduces first order terms into the non-divergence operator. The transformation law for these terms forces the boundary fixing diffeomorphism in the proof to be the identity. Positive scalar multiplication remains an ambiguity for a general coefficient matrix. Its boundary normalization follows from the prescribed boundary metric, not from the linear DN map alone.

\para{Outline of the proof}
The proof uses the first two derivatives of the nonlinear DN map.
\begin{itemize}
	\item \emph{The first linearized problem.} The metric is $g=D^2u_0$, and its inverse is the principal coefficient in \eqref{first linearized equation in the intro}. We write this equation as $(-\Delta_g+X_g\cdot\nabla)v=0$, where $X_g$ is given by \eqref{vector field X^b}. The planar inverse problem determines the metric, the vector field and their transformation by a boundary fixing diffeomorphism. The contracted Christoffel identity then gives a homogeneous Dirichlet problem for each component of the difference between that diffeomorphism and the identity. This determines $g$ up to a positive scalar factor in the original coordinates.
	
	\item \emph{The second linearized problem.} The second derivative of the DN map gives integrals of the form
	\begin{equation}\label{second integral id in introduction}
		\int_\Omega v^*\tr\big((D^2u_0)^{-1}D^2v^{(1)}(D^2u_0)^{-1}D^2v^{(2)}\big)\,dV_g.
	\end{equation}
	Here $v^{(1)},v^{(2)}$ solve the first linearized equation, and $v^*$ solves its transpose equation with respect to $dV_g$. We express the Euclidean Hessians in isothermal coordinates, including the first derivative terms caused by the coordinate change.
	
	We construct CGO solutions with holomorphic and antiholomorphic phases. For the two phases without critical points, we prove expansions of the correction terms to arbitrary order in $W^{2,p}$. For the phase with a critical point, we use the exact Cauchy integral equation for its correction. The terms in the second linearized identity are computed separately, including the terms in which derivatives fall on the amplitudes.
	
	The stationary phase limits yield an equation of the form
	\begin{equation}\label{eq:ucp_equation intro}
		\overline\partial(A_0\overline\partial G+\alpha G)+\delta G=\beta\big(\mathscr C(\gamma G)+P\big),
	\end{equation}
	where $A_0$ does not vanish, $\mathscr C$ is the scalar Cauchy operator, and $P$ is holomorphic. The coefficient multiplying the Cauchy integral also multiplies $P$. We prove unique continuation by introducing this entire parenthesized expression as another unknown and applying a Carleman estimate to the resulting first order system. The coefficients outside the original domain can be chosen so that the system vanishes on an exterior disk. This gives $G=0$ and hence $c=1$.
\end{itemize}
	
	For Hessian metrics, the conformal factor can also be determined directly from the compatibility of third derivatives; see Remark~\ref{remark: Hessian compatibility}. The second linearized argument proves the integral uniqueness result in Theorem~\ref{theorem: uniqueness of conformal factor} without using this compatibility.

	\subsection{Organization of the article}
	
	Section~\ref{sec: prel} gives the notation, the local solution map and the first two linearizations. Section~\ref{sec: boundary determination} determines the required boundary derivatives of the background solution. Section~\ref{sec: unique of diffeo} proves the fixed coordinate conformal uniqueness theorem. Section~\ref{sec: CGOs} constructs the CGO solutions and proves their noncritical phase expansions. Section~\ref{sec: Carleman and UCP} establishes unique continuation for the nonlocal equation. Section~\ref{sec: the second integral identity} derives the second linearized identity and its coordinate transformation. Section~\ref{sec: asymptotic analysis} computes its asymptotic limit. Section~\ref{sec: det of conformal factor} determines the conformal factor and completes the proof of Theorem~\ref{theorem: uniqueness}.

	\section{Preliminaries}\label{sec: prel}
	
	In this section, we will prepare several useful notations and tools for the study of the Monge-Ampère equation.
	
	\subsection{Notation, function spaces and matrix identities} 
	
	\subsubsection{Complex notation}
	
	We identify $\R^2$ with $\C$ and write $\nabla=(\partial_{x_1},\partial_{x_2})$. The complex derivatives are 
	\begin{equation}\label{definition of d and d-bar operators}
		\begin{split}
			\p=\p_z =\frac{\p}{\p z}=\frac{1}{2}\LC \p_{x_1} -\mathsf{i}\p_{x_2} \RC, \quad 	\overline{\p} =\p_{\overline{z}}=\frac{\p}{\p \overline{z}}=\frac{1}{2}\LC \p_{x_1} +\mathsf{i}\p_{x_2} \RC,
		\end{split}
	\end{equation}
	where $z=x_1+\mathsf{i}x_2$, $x_1,x_2\in\R$, and $\mathsf{i}^2=-1$. Write $\partial_k=\partial_{x_k}$ for $k=1,2$. Adding and subtracting the two identities in \eqref{definition of d and d-bar operators} gives
	\begin{equation}\label{d and d-bar relation}
		\begin{split}
			\p_{1}=\p+ \overline{\p}, \quad  \p_{2}=\mathsf{i}\big( \p - \overline{\p} \big),
		\end{split}
	\end{equation}
	and 
	\begin{equation}
		\begin{split}
			\overline{\p}^2 - \p^2 = \frac{1}{4} \big( \p_{1}^2 + 2\mathsf{i}\p_{x_1}\p_{x_2} - \p_{x_2}^2  -\p_{x_1}^2 + 2\mathsf{i}\p_{x_1}\p_{x_2} + \p_{x_2}^2  \big) = \mathsf{i}\p_{x_1}\p_{x_2}.
		\end{split}
	\end{equation}
	Note that the Hessian of any $C^2$ function $f=f(x)=f(x_1,x_2)$ can be written as 
	\begin{equation}\label{Hessian matrix}
		\begin{split}
			D^2 f=\left( \begin{matrix}
				\p_{11} & \p_{12} \\
				\p_{12} & \p_{22} \end{matrix} \right) f=\left( \begin{matrix}
				\LC \p +\overline{\p}\RC^2  & \mathsf{i}\LC \p^2 -\overline{\p}^2 \RC \\
				\mathsf{i}\LC \p^2 -\overline{\p}^2 \RC & -\LC \p -\overline{\p}\RC^2 \end{matrix} \right) f(z),
		\end{split}
	\end{equation}
	where we identify $z=x_1+\mathsf{i}x_2\in \C$.
	In particular, as $\Phi$ is holomorphic (i.e., $\overline{\p}\Phi=0$), one can obtain 
	\begin{equation}\label{Hessian_holo}
		\begin{split}
			D^2 \Phi=\left( \begin{matrix}
				\LC \p +\overline{\p}\RC^2  & \mathsf{i}\LC \p^2 -\overline{\p}^2 \RC \\
				\mathsf{i}\LC \p^2 -\overline{\p}^2 \RC & -\LC \p -\overline{\p}\RC^2 \end{matrix} \right) 
			\Phi= \left( \begin{matrix}
				1  & \mathsf{i} \\
				\mathsf{i}   & - 1\end{matrix} \right) \p^2 \Phi.
		\end{split}
	\end{equation}
	For an antiholomorphic function $\Psi$, so that $\partial\Psi=0$, we have 
	\begin{equation}\label{Hessian_antiholo}
		\begin{split}
			D^2 \Psi=\left( \begin{matrix}
				\LC \p +\overline{\p}\RC^2  & \mathsf{i}\LC \p^2 -\overline{\p}^2 \RC \\
				\mathsf{i}\LC \p^2 -\overline{\p}^2 \RC & -\LC \p -\overline{\p}\RC^2 \end{matrix} \right) 
			\Psi= \left( \begin{matrix}
				1  & -\mathsf{i} \\
				-\mathsf{i}   & - 1\end{matrix} \right) \overline{\p}^2 \Psi.
		\end{split}
	\end{equation}
	Therefore, the Hessian \eqref{Hessian matrix} can be written as 
	\begin{equation}\label{Hessian matrix in complex variable}
		\begin{split}
			D^2 f=\left( \begin{matrix}
				\LC \p +\overline{\p}\RC^2  & \mathsf{i}\LC \p^2 -\overline{\p}^2 \RC \\
				\mathsf{i}\LC \p^2 -\overline{\p}^2 \RC & -\LC \p -\overline{\p}\RC^2 \end{matrix} \right) f= A \p^2 f + B\overline{\p}^2 f + 2I_{2\times 2} 	\p \op  f,
		\end{split}
	\end{equation}
	where $A,B$ are matrices given by 	
	\begin{equation}\label{matrices A B}
		A:= \left( \begin{matrix}
			1  & \mathsf{i} \\
			\mathsf{i}   & - 1\end{matrix} \right) \quad \text{and} \quad B:= \left( \begin{matrix}
			1  & -\mathsf{i} \\
			-\mathsf{i}   & - 1\end{matrix} \right)
	\end{equation}
	are derived from \eqref{Hessian_holo} and \eqref{Hessian_antiholo}, and we used $4 \overline{\p}\p =\Delta$. Here, $I_{2\times 2}$ denotes the $2\times 2$ identity matrix. These notations will be used throughout the article. We use \eqref{Hessian matrix in complex variable} to expand the Hessian products in Section~\ref{sec: asymptotic analysis}.

\subsubsection{Function spaces}
We use the H\"older spaces $C^{k,\alpha}(\overline\Omega)$ and $C^{k,\alpha}(\partial\Omega)$, where $k$ is a nonnegative integer and $0<\alpha<1$. The boundary spaces are defined using smooth coordinate charts. These are Banach spaces and algebras. The Leibniz rule and $[fg]_\alpha\leq\|f\|_{L^\infty}[g]_\alpha+\|g\|_{L^\infty}[f]_\alpha$ give $\|fg\|_{C^{k,\alpha}}\leq C\|f\|_{C^{k,\alpha}}\|g\|_{C^{k,\alpha}}$. We equip the smooth function spaces with their usual Fr\'echet topologies. The family of $C^{m,\alpha}$ norms gives the same topology, since $\|f\|_{C^m}\leq C\|f\|_{C^{m,\alpha}}$ and $\|f\|_{C^{m,\alpha}}\leq C_m\|f\|_{C^{m+1}}$.
	
	\subsubsection{Matrix identities}
	
	We also collect several useful properties for matrix computations. Let us first recall the Jacobi formula for a matrix, which is given by 
	\begin{equation}\label{Jacobi formula}
		\frac{d}{dt}\det A(t)= (\det A(t)) \tr \Big( A(t)^{-1}\frac{d A(t)}{dt}\Big),
	\end{equation}
	for any differentiable family of invertible $n\times n$ matrices $A(t)$. Moreover, it also holds that 
	\begin{equation}\label{derivative of matrix A}
		\frac{d}{dt}A^{-1}(t)=-A^{-1}(t)A'(t)A^{-1}(t),
	\end{equation}
	for any differentiable family of invertible matrices $A(t)$. We use these identities to differentiate the Monge--Amp\`ere equation.

\subsection{Well-posedness}\label{sec: well}

We work on the convex solution branch. For the implicit function theorem, let $\mathcal B_\delta(\partial\Omega)$ be the open ball of radius $\delta$ in the real Banach space $C^{4,\alpha}(\partial\Omega)$. Then $B_\delta(\partial\Omega)=\mathcal B_\delta(\partial\Omega)\cap C^\infty(\partial\Omega)$. Define $Q(u)=\det D^2u$.

\begin{proposition}[Well-posedness]\label{prop: well-posed}
	Let $\Omega\subset\R^2$ be bounded and uniformly convex, with smooth boundary. Fix $\alpha\in(0,1)$ and $F\in C^{4,\alpha}(\overline\Omega)$ with $F\geq c_0>0$.
	\begin{enumerate}[\rm(i)]
		\item\label{item 1 well-posed} There is a unique convex solution $u_0\in C^{4,\alpha}(\overline\Omega)$ of
		\begin{align}\label{MA equation in well-posed zero BC}
			\begin{cases}\det D^2u_0=F&\text{in }\Omega,\\u_0=0&\text{on }\partial\Omega.\end{cases}
		\end{align}
		Its Hessian satisfies $\lambda_0I_2\leq D^2u_0\leq\Lambda_0I_2$ on $\overline\Omega$ for positive constants $\lambda_0,\Lambda_0$.
		
		\item\label{item 2 well-posed} There are $\delta,C>0$ such that every $\varphi\in\mathcal B_\delta(\partial\Omega)$ gives a convex solution $u_\varphi\in C^{4,\alpha}(\overline\Omega)$ of
		\begin{align}\label{MA equation in well-posed}
			\begin{cases}\det D^2u_\varphi=F&\text{in }\Omega,\\u_\varphi=\varphi&\text{on }\partial\Omega,\end{cases}
		\end{align}
		and
		\begin{equation}\label{continuous estimate in well-posed}
			\|u_\varphi-u_0\|_{C^{4,\alpha}(\overline\Omega)}\leq C\|\varphi\|_{C^{4,\alpha}(\partial\Omega)}.
		\end{equation}
		The solution is unique among functions $u$ with $\|u-u_0\|_{C^{4,\alpha}}\leq C\delta$. The constants can be chosen so that $D^2u_\varphi\geq(\lambda_0/2)I_2$ throughout the branch.
		
		\item\label{item 3 well-posed} The maps $S:\mathcal B_\delta(\partial\Omega)\to C^{4,\alpha}(\overline\Omega)$, $S(\varphi)=u_\varphi$, and $\Lambda_F:\mathcal B_\delta(\partial\Omega)\to C^{3,\alpha}(\partial\Omega)$ are $C^\infty$ maps between Banach spaces. If $F$ and $\varphi$ are smooth, then $u_\varphi$ is smooth, and the restrictions
		\begin{equation}\label{solution map}
			\begin{aligned}
				S&:B_\delta(\partial\Omega)\to C^\infty(\overline\Omega),&S(\varphi)&=u_\varphi,\\
				\Lambda_F&:B_\delta(\partial\Omega)\to C^\infty(\partial\Omega),&\Lambda_F(\varphi)&=\partial_\nu u_\varphi|_{\partial\Omega}
			\end{aligned}
		\end{equation}
		are smooth: their iterated directional derivatives exist and are jointly continuous in the base point and the directions.
	\end{enumerate}
\end{proposition}

\begin{proof}
	For \ref{item 1 well-posed}, the classical Dirichlet theory on a uniformly convex domain gives the unique convex solution with zero boundary data; see \cite[Chapter 3]{Figalli17}. Convexity implies that its Hessian is positive semidefinite in $\Omega$, and continuity gives the same assertion on $\overline\Omega$. The equation extends to the boundary, so $\det D^2u_0=F\geq c_0$ there as well. If $\Lambda_0$ bounds the largest eigenvalue on the closed domain, the smaller eigenvalue is at least $c_0/\Lambda_0$. This gives the stated uniform bounds. The same Dirichlet regularity result gives $u_0\in C^{k+2,\alpha}$ when $F\in C^{k,\alpha}$, $k\geq2$, with the corresponding regularity of the boundary.
	
	For \ref{item 2 well-posed}, set $X=C^{4,\alpha}(\partial\Omega)$, $Y=C^{4,\alpha}(\overline\Omega)$ and $Z=C^{2,\alpha}(\overline\Omega)\times C^{4,\alpha}(\partial\Omega)$. Write $T_0u=u|_{\partial\Omega}$ and define
	\begin{equation}\label{mapping prop in well-posed}
		\Phi:X\times Y\to Z,\quad \Phi(\varphi,u)=(Q(u)-F,T_0u-\varphi).
	\end{equation}
	Since $Q(u)=u_{11}u_{22}-u_{12}^2$, the Hessian map and the H\"older algebra estimate give
	\begin{equation}\label{polynomial of determinant}
		Q:C^{4,\alpha}(\overline\Omega)\to C^{2,\alpha}(\overline\Omega).
	\end{equation}
	More explicitly,
	\[
	Q(u+v)=Q(u)+u_{22}v_{11}+u_{11}v_{22}-2u_{12}v_{12}+v_{11}v_{22}-v_{12}^2.
	\]
	Thus $DQ(u)v=u_{22}v_{11}+u_{11}v_{22}-2u_{12}v_{12}$, with a remainder bounded by $C\|v\|_{C^{4,\alpha}}^2$ in $C^{2,\alpha}$. Its second derivative is the continuous bilinear map $D^2Q(u)[v,w]=v_{11}w_{22}+w_{11}v_{22}-2v_{12}w_{12}$, and all higher derivatives vanish. Consequently, $\Phi$ is $C^\infty$ in the Banach space sense and $\Phi(0,u_0)=0$.
	
	Jacobi's formula gives
	\begin{equation}\label{eq: wellposed-linearized-map}
		D_u\Phi(0,u_0)v=(P_0v,T_0v),\quad
		P_0v=F\tr((D^2u_0)^{-1}D^2v).
	\end{equation}
	The coefficient matrix is $A_0=F(D^2u_0)^{-1}=\Cof(D^2u_0)\in C^{2,\alpha}$. Its eigenvalues are bounded above and below by positive constants, by \ref{item 1 well-posed} and $F\geq c_0$. The map in \eqref{eq: wellposed-linearized-map} is therefore bounded from $Y$ to $Z$.
	
	To prove that it is an isomorphism, let $(G,\phi)\in Z$ and consider
	\begin{equation}\label{equ Schauder in well-posed}
		\begin{cases}
			F\tr((D^2u_0)^{-1}D^2v)=G&\text{in }\Omega,\\
			v=\phi&\text{on }\partial\Omega.
		\end{cases}
	\end{equation}
	Since $1/F\in C^{2,\alpha}$, this is equivalent to
	\begin{equation}\label{equ Schauder in well-posed 2}
		\begin{cases}
			\tr((D^2u_0)^{-1}D^2v)=G/F&\text{in }\Omega,\\
			v=\phi&\text{on }\partial\Omega.
		\end{cases}
	\end{equation}
	The linear Dirichlet theory gives a unique $C^{2,\alpha}$ solution, and the global Schauder estimate improves it to $C^{4,\alpha}$; see \cite[Chapter 6]{gilbarg2015elliptic}. The estimate initially includes $\|v\|_{C^0}$. The maximum principle estimate bounds this term by $C(\|G\|_{C^0}+\|\phi\|_{C^0})$, so
	\begin{equation}\label{eq: wellposed-inverse-estimate}
		\|v\|_{C^{4,\alpha}}\leq C\big(\|G\|_{C^{2,\alpha}}+\|\phi\|_{C^{4,\alpha}(\partial\Omega)}\big).
	\end{equation}
	This proves surjectivity and boundedness of the inverse. The maximum principle with zero data proves injectivity.
	
	The Banach implicit function theorem \cite{renardy2006introduction} gives open neighborhoods $\mathcal U$ of zero in $X$ and $\mathcal V$ of $u_0$ in $Y$, and a smooth map $S:\mathcal U\to\mathcal V$ with $S(0)=u_0$ and $\Phi(\varphi,S(\varphi))=0$. Within $\mathcal U\times\mathcal V$, these are all solutions of $\Phi(\varphi,u)=0$.
	
	Choose $r>0$ so that the closed $Y$-ball of radius $r$ about $u_0$ lies in $\mathcal V$ and every function in that ball has Hessian at least $(\lambda_0/2)I_2$. This is possible by the continuous inclusion $Y\subset C^2$. Continuity of $DS$ at zero gives $\delta_0,C_1>0$ such that $\mathcal B_{\delta_0}\subset\mathcal U$ and $\|DS(\varphi)\|_{\mathcal L(X,Y)}\leq C_1$ there. Integrating the derivative on the line segment from zero to $\varphi$ gives
	\[
	S(\varphi)-u_0=\int_0^1DS(t\varphi)[\varphi]\,dt.
	\]
	Take $C\geq\max\{C_1,1\}$ and choose $\delta<\delta_0$ with $C\delta<r$. This proves \eqref{continuous estimate in well-posed}, the asserted positivity, and uniqueness in the stated class.
	
	We prove the spatial smoothness in \ref{item 3 well-posed} before discussing smooth dependence in the higher spaces. Assume $F,\varphi$ are smooth and let $u=S(\varphi)$. If $u\in C^{m,\alpha}$, $m\geq4$, put $P_uv=\Cof(D^2u)^{ij}\partial_{ij}v$. For a smooth vector field $T=T^k\partial_k$ tangent to the boundary, differentiation gives
	\begin{equation}\label{eq: wellposed-tangential-bootstrap}
		P_u(Tu)=TF+2\Cof(D^2u)^{ij}(\partial_iT^k)\partial_{jk}u+\Cof(D^2u)^{ij}(\partial_{ij}T^k)\partial_ku.
	\end{equation}
	The right-hand side is $C^{m-2,\alpha}$ and $(Tu)|_{\partial\Omega}=T\varphi$ is smooth. The global Schauder estimate gives $Tu\in C^{m,\alpha}$. Vector fields supported in the interior give the interior improvement.
	
	Near a boundary point, write the boundary as $x_2=\eta(x_1)$ and set $v(s,t)=u(s,\eta(s)+t)$. A smooth tangent vector field agreeing locally with $\partial_{x_1}+\eta'\partial_{x_2}$ gives $v_s\in C^{m,\alpha}$. The chain rule and the equation give
	\begin{equation}\label{eq: wellposed-normal-bootstrap}
		(v_{ss}-\eta''v_t)v_{tt}-v_{st}^2=F(s,\eta(s)+t).
	\end{equation}
	The coefficient $v_{ss}-\eta''v_t=(1,\eta')D^2u(1,\eta')^T$ is bounded below by a positive constant. Therefore $v_{tt}=(F(s,\eta(s)+t)+v_{st}^2)/(v_{ss}-\eta''v_t)$ belongs to $C^{m-1,\alpha}$. Together with $v_s\in C^{m,\alpha}$, this gives every derivative of order $m+1$: those containing an $s$ derivative come from $v_s$, and the pure $t$ derivative comes from $v_{tt}$. A finite boundary covering and the interior estimates give $u\in C^{m+1,\alpha}(\overline\Omega)$. Induction proves $u\in C^\infty(\overline\Omega)$.
	
	Fix now a smooth $\varphi_0\in B_\delta$ and put $u_*=S(\varphi_0)$. For each $m\geq4$, set $X_m=C^{m,\alpha}(\partial\Omega)$, $Y_m=C^{m,\alpha}(\overline\Omega)$, and $Z_m=C^{m-2,\alpha}(\overline\Omega)\times X_m$. The formula \eqref{mapping prop in well-posed} defines a smooth map $\Phi_m:X_m\times Y_m\to Z_m$. Its derivative at $(\varphi_0,u_*)$ is the linearized Dirichlet map with coefficient $F(D^2u_*)^{-1}$. The argument for \eqref{eq: wellposed-inverse-estimate}, at order $m$, proves that this derivative is an isomorphism.
	
	Therefore, the implicit function theorem gives a smooth local map $S_m$ with values in $Y_m$. Shrink its $X_m$ neighborhood so that its boundary values lie in $\mathcal B_\delta$ and its solutions lie in $\mathcal V$. Uniqueness of the original branch then gives $S_m=S$ on that neighborhood. Thus the maps at different levels describe the same solutions. No common radius in all higher norms is required.
	
	Their derivatives agree after inclusion into $Y$, by uniqueness of Banach derivatives. For smooth directions, each iterated derivative of $S$ consequently belongs to every $Y_m$. Joint continuity in a prescribed target seminorm follows from the corresponding Banach map $S_m$. This proves the smoothness statement for $S$ in \eqref{solution map}.
	
	Finally, $T_\nu u=\nu\cdot\nabla u|_{\partial\Omega}$ is a continuous linear map $Y_m\to C^{m-1,\alpha}(\partial\Omega)$. Since $\Lambda_F=T_\nu S$, the chain rule gives
	\begin{equation}\label{eq: smooth-DN-derivatives}
		D^j\Lambda_F(\varphi)[\psi_1,\ldots,\psi_j]
		=T_\nu\big(D^jS(\varphi)[\psi_1,\ldots,\psi_j]\big).
	\end{equation}
	At $m=4$ this proves the Banach space statement, without assuming smooth $F$. At all higher levels it proves the smooth-space statement. The proposition follows.
\end{proof}

	\subsection{Higher order linearization}\label{sec: higher order linear}
	
	Let $\phi_1,\phi_2\in C^\infty(\partial\Omega)$ be real-valued and set
	\begin{equation}\label{eps-Dirichlet data}
		\varphi_\eps=\eps_1\phi_1+\eps_2\phi_2.
	\end{equation}
	For sufficiently small real $\eps_1,\eps_2$, these values lie in $B_\delta(\partial\Omega)$. Proposition~\ref{prop: well-posed} shows that $u_\eps=S(\varphi_\eps)$ is smooth in the parameters with values in $C^{4,\alpha}$, and, for smooth $F$, with values in $C^\infty$.
	
	The Hessian, Dirichlet trace and normal trace are continuous linear maps on the relevant H\"older spaces. Hence,
	\begin{equation}\label{eq: parameter-differentiation-traces}
		\partial_\eps^\beta D^2u_\eps=D^2\partial_\eps^\beta u_\eps,\quad
		\partial_\eps^\beta T_0u_\eps=T_0\partial_\eps^\beta u_\eps,\quad
		\partial_\eps^\beta T_\nu u_\eps=T_\nu\partial_\eps^\beta u_\eps.
	\end{equation}
	The smoothness of $Q$ justifies differentiating $Q(u_\eps)=F$ in $C^{2,\alpha}$. In particular, $v^{(k)}=DS(0)[\phi_k]$ and $w=D^2S(0)[\phi_1,\phi_2]$ have boundary values $\phi_k$ and zero, respectively. Moreover,
	\begin{equation}\label{eq: measured-first-second-variations}
		D\Lambda_F(0)[\phi_k]=\partial_\nu v^{(k)}|_{\partial\Omega},\quad
		D^2\Lambda_F(0)[\phi_1,\phi_2]=\partial_\nu w|_{\partial\Omega}.
	\end{equation}
	We first differentiate in real directions. The resulting linear and bilinear identities extend by complex multilinearity, without requiring a complex convex solution branch.

	\subsubsection{The first linearization}\label{subsec: first linearization}
	
	Since $F$ is independent of $\eps$, differentiation of \eqref{MA equation} with respect to $\eps_k$ and Jacobi's formula \eqref{Jacobi formula} give
	\begin{align}\label{Linearized MA equation in eps}
		\begin{cases}
			(\det D^2u_\eps)\tr((D^2u_\eps)^{-1}D^2\partial_{\eps_k}u_\eps)=0&\text{in }\Omega,\\
			\partial_{\eps_k}u_\eps=\phi_k&\text{on }\partial\Omega,
		\end{cases}
	\end{align}
	for $k=1,2$. Setting $\eps=0$ gives
	\begin{align}\label{Linearized MA equation}
		\begin{cases}
			\big( \det D^2 u_0\big) \tr \big(\big( D^2 u_0\big)^{-1}D^2v^{(k)}\big)=0 &\text{ in }\Omega,\\
			v^{(k)}=\phi_k & \text{ on }\p \Omega.
		\end{cases}
	\end{align}
	where
	\begin{align*}
		u_0 =u(x;0),\quad \text{ and }\quad v^{(k)}=\left. \p_{\eps_k} \right|_{\eps=0}u_\eps ,
	\end{align*}
	for $k=1,2$.
	Moreover, it is easy to see that $u_0$ is the solution to \eqref{MA equation} with zero boundary data, i.e.,
	\begin{align}\label{MA equation zero boundary}
		\begin{cases}
			\det D^2 u_0 =F &\text{ in }\Omega, \\
			u_0=0 &\text{ on }\p \Omega.
		\end{cases}
	\end{align}
	Now, by plugging \eqref{MA equation zero boundary} into \eqref{Linearized MA equation} and using $F>0$ in $\Omega$, we obtain a \emph{linear} second order elliptic equation
	\begin{align}\label{Linearized MA equation2}
		\begin{cases}
			u_0^{ab}\p_{ab}v^{(k)}=0 &\text{ in }\Omega,\\
			v^{(k)}=\phi_k & \text{ on }\p \Omega,
		\end{cases}
	\end{align}
	which is of the \emph{non-divergence form} for $k=1,2$, where $\big( u_0^{ab}\big)_{1\leq a,b\leq 2} = \LC D^2u_0 \RC^{-1}$. (Throughout this work, we employ the Einstein summation convention: repeated indices imply summation over the values $1,2$.) The first variation of the nonlinear DN map gives the normal derivative of $v^{(k)}$ by \eqref{eq: measured-first-second-variations}. Together with its prescribed boundary value, this determines its Cauchy data. The background Hessian is positive definite by Proposition~\ref{prop: well-posed}; no positivity assumption is imposed on the directions $\phi_k$. Determination of $(D^2u_0)^{-1}$ would determine $D^2u_0$ and its determinant $F$.
	
	In what follows, let us use the notation 
	\begin{equation}\label{geometry from u_0}
		g^{ab}:= u_0^{ab} \text{ in }\Omega, \text{ for }a,b=1,2,
	\end{equation}
	then one can derive 
	\begin{equation}\label{comp first linearization}
		\begin{split}
			g^{ab}\p_{ab}v^{(k)}=0 \iff \sqrt{\abs{g}}g^{ab}\p_{ab}v^{(k)} =0 \iff \p_a \big( \sqrt{\abs{g}}g^{ab}\p_b v^{(k)} \big) - \p_a \big( \sqrt{\abs{g}}g^{ab} \big) \p_b v^{(k)} =0,
		\end{split}
	\end{equation}
	for $k=1,2$, where we used $\abs{g}=\det (g)=\det (g_{ab})>0$ in $\Omega$.
	Hence, we can rewrite \eqref{Linearized MA equation2} into 
	\begin{equation}\label{Linearized main equation}
		\begin{cases}
			\LC -\Delta_g  +X_g \cdot \nabla \RC   v^{(k)} =0 &\text{ in }\Omega, \\
			v^{(k)}=\phi_k&\text{ on }\p \Omega,
		\end{cases}
	\end{equation}
	for $k=1,2$, where 
	\begin{equation}\label{Laplace-Beltrami op.}
		\Delta_g=\frac1{\sqrt{|g|}}\partial_a(\sqrt{|g|}g^{ab}\partial_b)=g^{ab}\partial_{ab}+X_g^b\partial_b.
	\end{equation}
	stands for the Laplace-Beltrami operator, and 
	\begin{equation}\label{vector field X^b}
		\begin{split}
			X_g& =\big( X_g^1(x), X_g^2(x)\big):\overline{\Omega}\to \R^2,\\
			X_g^b&:=\frac{1}{\sqrt{\abs{g}}}\sum_{a=1}^{2}\p_a \big( \sqrt{\abs{g}}g^{ab}\big), \quad \text{for}\quad b=1,2,
		\end{split}
	\end{equation}  
	which denotes the vector-valued coefficient of the first order term. Consequently, $\mathcal L_g:=-\Delta_g+X_g\cdot\nabla=-g^{ab}\partial_{ab}$. We use these conventions throughout the paper.
	
	\subsubsection{The second linearization}\label{subsec: second linearization}
	
	Consider the Dirichlet data of the form \eqref{eps-Dirichlet data} in \eqref{MA equation}, and let us rewrite \eqref{Linearized MA equation in eps} as 
	\begin{equation}\label{Linearized MA equation in eps 2}
		\begin{cases}
			\tr \big(\big( D^2 u_{\eps}\big)^{-1}D^2 \big( \p_{\eps_k} u_\eps\big) \big)=0 &\text{ in }\Omega,\\
			\p_{\eps_k}u_\eps=\phi_k &\text{on }\p\Omega,
		\end{cases}
	\end{equation} 
	for $k=1,2$, where we used $\det D^2 u_{\eps}=F>0$ in $\Omega$. Differentiating the equation for $k=2$ once more with respect to $\eps_1$ gives 
	\begin{equation}\label{2nd linear comp}
		\begin{aligned}
			0&=\partial_{\eps_1}\tr((D^2u_\eps)^{-1}D^2\partial_{\eps_2}u_\eps)\\
			&=\tr(\partial_{\eps_1}((D^2u_\eps)^{-1})D^2\partial_{\eps_2}u_\eps)+\tr((D^2u_\eps)^{-1}D^2\partial_{\eps_1}\partial_{\eps_2}u_\eps)\\
			&=-\tr((D^2u_\eps)^{-1}D^2\partial_{\eps_1}u_\eps(D^2u_\eps)^{-1}D^2\partial_{\eps_2}u_\eps)+\tr((D^2u_\eps)^{-1}D^2\partial_{\eps_1}\partial_{\eps_2}u_\eps).
		\end{aligned}
	\end{equation}
	where we used the fact \eqref{derivative of matrix A}. Inserting $\eps=0$ into \eqref{2nd linear comp}, we can obtain the second linearized equation as 
	\begin{equation}\label{2nd Linearized MA equation}
		\begin{cases}
			\tr \big( \big( D^2u_0\big)^{-1}  D^2 w\big) = \tr  \big( \big( D^2 u_0\big)^{-1} D^2  v^{(1)} \big( D^2 u_0\big)^{-1} D^2v^{(2)}\big) &\text{ in }\Omega, \\
			w= 0&\text{ on }\p \Omega,
		\end{cases}
	\end{equation}
	where $w= \left. \p_{\eps_1\eps_2}^2 \right|_{\eps=0} u_{\eps} $, and we utilized $\det \LC D^2u_\eps \RC =F >0$ in $\Omega$. Similar to the first linearized equation \eqref{Linearized MA equation2}, we can rewrite \eqref{2nd Linearized MA equation} as 
	\begin{equation}\label{2nd Linearized MA equation 2}
		\begin{cases}
			(-\Delta_g+X_g\cdot\nabla)w=-\tr(g^{-1}D^2v^{(1)}g^{-1}D^2v^{(2)})&\text{in }\Omega,\\
			w=0&\text{on }\partial\Omega.
		\end{cases}
	\end{equation}
	where $g$ and $X_g$ are given by \eqref{geometry from u_0} and \eqref{vector field X^b}, respectively. Here $v^{(k)}$ solves \eqref{Linearized MA equation2} for $k=1,2$.

	\section{Boundary determination}\label{sec: boundary determination}
	
	We determine the boundary derivatives of the zero boundary solution through order four. The data are $\Lambda_F(0)$ and the prescribed boundary derivatives of $F$ through order two.
	
	\begin{lemma}[Boundary determination]\label{lemma: boundary determination}
		Adopting all assumptions in Theorem \ref{theorem: uniqueness}, let $u_0\in C^{\infty}(\overline{\Omega})$ be the solution to \eqref{MA equation zero boundary}. Suppose that $F$ is known up to second order on the boundary, then $D^\beta u_0\big|_{\p \Omega}$ can be determined by $\Lambda_F(0)$, where $\beta=(\beta_1,\beta_2)\in (\N \cup \{0\})^2$ is a multi-index, with $\abs{\beta}=\beta_1+\beta_2\leq 4$. 
		
		In other words, let $F_1,F_2 \in C^{\infty}(\overline{\Omega})$ be positive sources, and suppose that $F_1$ and $F_2$ agree up to second order on the boundary, then 
		\begin{equation}\label{eq: same DN map zero}
			\Lambda_{F_1}(0)=\Lambda_{F_2}(0) \text{ on }\p \Omega
		\end{equation}
		implies $D^\beta u_0^{(1)}\big|_{\p \Omega}=D^\beta u_0^{(2)}\big|_{\p \Omega}$, for all $\abs{\beta}\leq 4$, where $u_0^{(j)}\in C^{\infty}(\overline{\Omega})$ is the solution to 
		\begin{align}\label{MA equation zero boundary j=1,2}
			\begin{cases}
				\det D^2 u_0^{(j)} =F_j &\text{ in }\Omega, \\
				u_0^{(j)}=0 &\text{ on }\p \Omega,
			\end{cases}
		\end{align}
		for $j=1,2$.
	\end{lemma}

	\begin{proof}
		We aim to show that for any point $ x_0 \in \partial \Omega $ and any multi-index $ \beta $ with $ |\beta| \leq 4 $, the derivative $ D^\beta u_0(x_0) $ can be determined using the data $ F|_{\partial \Omega} $ with its boundary derivatives up to order two, and the DN map $ \Lambda_F(0) $. After a rotation and translation, assume $x_0=0$ and write the boundary locally as $x_2=\varphi(x_1)$, with $\Omega$ on the side $x_2>\varphi(x_1)$. Here $\varphi$ is smooth and convex, and $\varphi(0)=\varphi'(0)=0$. In particular, $\nu(0)=-e_2$.
		
		Proposition~\ref{prop: well-posed} gives uniform positivity of $D^2u_0$ on $\overline\Omega$. This includes the boundary: convexity gives a positive semidefinite Hessian there by continuity, while its determinant equals $F\geq c_0$. Thus, in the fixed rotated coordinates,
		\begin{equation}\label{eq:positive_11}
			\partial_{11}u_0\geq\lambda_0>0\quad\text{on }\overline\Omega.
		\end{equation}
		Since $u_0=0$ on $\partial\Omega$, its tangential derivatives vanish and $\nabla u_0=\nu\Lambda_F(0)$ there. Thus, the full first boundary jet is known; in particular, $\partial_2u_0(0)=-\Lambda_F(0)(0)$.

		Since $u_0 = 0$ on $\partial \Omega$, we have $u_0(x_1, \varphi(x_1)) = 0$ for all $x_1 \in (-\delta, \delta)$, as well as its tangential derivatives. Thus, differentiating this identity twice with respect to $x_1$, we can compute
		\begin{equation}\label{eq:boundary_identity}
			\begin{aligned}
				0 &= \underbrace{\frac{d^2}{dx_1^2}\Big|_{x_1=0} u_0(x_1, \varphi(x_1))}_{\text{tangential derivative}} \\
				&= \big(\partial_{11} u_0\big)(0) + 2\big(\partial_{12} u_0\big)(0) \varphi'(0) + \big(\partial_{22} u_0\big)(0) \big(\varphi'(0)\big)^2 + 
				\big(\partial_2 u_0\big)(0) \varphi''(0).
			\end{aligned}
		\end{equation}
		Using $\varphi'(0) = 0$, this simplifies to
		\begin{equation}\label{eq:u11_u2_relation}
			\big(\partial_{11} u_0\big)(0) + \big(\partial_2 u_0\big)(0) \varphi''(0) = 0.
		\end{equation}
		Hence,
		\begin{equation}\label{eq:u11_known}
			\big(\partial_{11} u_0\big)(0) = -\big(\partial_2 u_0\big)(0) \varphi''(0).
		\end{equation}
		The right-hand side is known: $ \partial_2 u_0(0) $ is obtained from the DN map $\Lambda_F$ at $x_0=0\in \p\Omega$, and $\varphi''(0)$ is the curvature of the boundary at $x_0 = 0$, which is computable from the parametrization. Thus, $\big(\partial_{11} u_0\big)(0)$ is determined.

		The known boundary gradient determines the trace of $\partial_2u_0$. Differentiating $x_1\mapsto\partial_2u_0(x_1,\varphi(x_1))$ at zero gives $\partial_{12}u_0(0)+\varphi'(0)\partial_{22}u_0(0)=\partial_{12}u_0(0)$, so this mixed derivative is known. Moreover, the Monge–Ampère equation \eqref{MA equation zero boundary} gives:
		\begin{equation}
			\big(\partial_{11} u_0\big)(0) \big(\partial_{22} u_0\big)(0) - \left( \partial_{12} u_0 \right)^2(0) = F(0),
		\end{equation}
		which yields
		\begin{equation}\label{eq:u22_known}
			\big(\partial_{22} u_0\big)(0) = \frac{F(0) + \left( \partial_{12} u_0 \right)^2(0)}{\big(\partial_{11} u_0\big)(0)}.
		\end{equation}
		This is valid due to the positivity of $\big(\partial_{11} u_0\big)(0)$. Therefore, all second-order derivatives $\big(\partial_{ij} u_0\big)(0)$ with $ i,j \in \{1,2\}$ are now determined. Since $x_0 \in \partial \Omega$ was arbitrary, this argument applies uniformly along $\partial \Omega$, and we conclude that:
		\[
		\left. \partial_{11} u_0 \right|_{\partial \Omega},\quad \left. \partial_{12} u_0 \right|_{\partial \Omega}, \quad \left. \partial_{22} u_0 \right|_{\partial \Omega}
		\]
		are determined.
		
		For third derivatives, take tangential derivatives of the known traces of $\partial_{11}u_0$, $\partial_{12}u_0$ and $\partial_{22}u_0$. Since $\varphi'(0)=0$, these give $\partial_{111}u_0(0)$, $\partial_{112}u_0(0)$ and $\partial_{122}u_0(0)$, respectively. Differentiating the Monge--Amp\`ere equation with respect to $x_2$ gives the pointwise identity
		\begin{equation}\label{eq:ma_diff}
			\partial_{11}u_0\,\partial_{222}u_0+\partial_{112}u_0\,\partial_{22}u_0-2\partial_{12}u_0\,\partial_{122}u_0=\partial_2F\quad\text{in }\Omega.
		\end{equation}
		Evaluate this identity at zero. The quantity $\partial_2F(0)=-\partial_\nu F(0)$ is prescribed, and all terms except $\partial_{222}u_0(0)$ are known. Solving for $ \partial_{222} u_0(0) $, we obtain:
		\begin{equation}\label{eq:u222_known}
			\partial_{222} u_0(0) = \frac{\partial_2 F(0) - \partial_{211} u_0(0) \partial_{22} u_0(0) + 2\partial_{12} u_0(0) \partial_{122} u_0(0)}{\partial_{11} u_0(0)}.
		\end{equation}
		Again, this is valid due to \eqref{eq:positive_11}. As a result, all third-order derivatives $ \partial_{abc} u_0(0)$, with $a,b,c \in \{1,2\}$, are now known. 
		
		The same argument at each boundary point determines the traces of all third derivatives. For the remaining calculation, write $u_{i_1\cdots i_k}=\partial_{i_1}\cdots\partial_{i_k}u_0$ and $F_{22}=\partial_{22}F$. Tangential differentiation of $u_{111},u_{112},u_{122},u_{222}$ now determines $u_{1111}(0),u_{1112}(0),u_{1122}(0),u_{1222}(0)$, respectively. The additional transverse term in each derivative is multiplied by $\varphi'(0)=0$.
		
		Differentiate \eqref{eq:ma_diff} once more in $x_2$. This gives
		\begin{equation}\label{eq: fourth-boundary-identity}
			u_{1122}u_{22}+2u_{112}u_{222}+u_{11}u_{2222}-2u_{122}^2-2u_{12}u_{1222}=F_{22}.
		\end{equation}
		All terms at the origin except $u_{2222}(0)$ are known. Since $u_{11}(0)>0$,
		\begin{equation}\label{eq:u2222_known}
			u_{2222}(0)=\frac{F_{22}(0)-u_{1122}(0)u_{22}(0)-2u_{112}(0)u_{222}(0)+2u_{122}(0)^2+2u_{12}(0)u_{1222}(0)}{u_{11}(0)}.
		\end{equation}
		These are all fourth derivatives in two dimensions. The source derivatives in the rotated coordinates are determined by its prescribed Cartesian jets. Rotating back determines the Cartesian derivatives of $u_0$. This proves the assertion through order four and the pairwise conclusion.
		\end{proof}

	\section{Unique determination of the metric}\label{sec: unique of diffeo}
	
	Recall that our goal is to recover the Hessian of $u_0$ in $\Omega$, where $u_0$ is the solution to~\eqref{MA equation zero boundary}. Once the Hessian $D^2 u_0$ is determined, the source function $F$ is also fully determined. Let us adopt the notation introduced in~\eqref{geometry from u_0}. Lemma~\ref{lemma: boundary determination} gives $g|_{\partial\Omega}$. Together with $D\Lambda_F(0)$, this boundary metric will determine the conormal DN map of the first linearized equation. Unless otherwise specified, we now allow $g$ to be any smooth positive definite symmetric matrix field on $\overline\Omega$. In particular, while in the Monge--Amp\`ere setting $g = D^2 u_0$, the analysis of the linearized inverse problem applies to arbitrary smooth metrics. 
	
	We use $\partial_{\nu_g}v=\sqrt{|g|}\,\nu_ag^{ab}\partial_bv$. This is a conormal density, so its boundary integrals are taken with respect to Euclidean arclength without an additional factor $\sqrt{|g|}$. Using the first linearized equation of the Monge-Amp\`ere equation (see Section~\ref{sec: higher order linear}), we define the DN map $\Lambda'_g$ corresponding to the boundary value problem
	\begin{equation}\label{Linearized main equation in DN}
		\begin{cases}
			\left(-\Delta_g + X_g \cdot \nabla\right)v = 0 & \text{in } \Omega, \\
			v = \phi & \text{on } \partial \Omega,
		\end{cases}
	\end{equation}
	as
	\begin{equation}\label{DN_linearized}
		\Lambda'_g: C^{\infty}(\partial \Omega) \to C^{\infty}(\partial \Omega), \quad \phi \mapsto \partial_{\nu_g} v_\phi,
	\end{equation}
	where $v_\phi \in C^{\infty}(\overline{\Omega})$ denotes the solution to~\eqref{Linearized main equation in DN}. 
	
	To proceed, we prove the following result, which shows that the DN map of the fully nonlinear Monge-Amp\`ere equation determines the DN map of its linearized counterpart~\eqref{Linearized main equation in DN}.

	\begin{lemma}\label{lemma: nonlinear to linear DN}
		Assume the hypotheses of Theorem~\ref{theorem: uniqueness}. The DN map $\Lambda_F$ of \eqref{MA equation} (see the definition \eqref{DN map}) determines the DN map $\Lambda'_g$ of \eqref{DN_linearized}, where $g=D^2u_0$ is the Hessian of $u_0$ ($u_0$ is the solution to \eqref{MA equation zero boundary}).
	\end{lemma}
	
	\begin{proof}
		Lemma~\ref{lemma: boundary determination}, together with the prescribed source jets, gives $g|_{\partial\Omega}$. If $\tau$ is a unit tangent, the boundary values and measured normal derivative give
		\begin{equation}\label{boundary determination of solution to MA}
			\nabla u_\varphi|_{\partial\Omega}=\tau\partial_\tau\varphi+\nu\Lambda_F(\varphi).
		\end{equation}
		Parameter differentiation commutes with the traces by \eqref{eq: parameter-differentiation-traces}. Hence, $\nabla v_\phi|_{\partial\Omega}=\tau\partial_\tau\phi+\nu D\Lambda_F(0)[\phi]$. It follows that
		\[
		\Lambda'_g\phi=\sqrt{|g|}\big((\nu^Tg^{-1}\tau)\partial_\tau\phi+(\nu^Tg^{-1}\nu)D\Lambda_F(0)[\phi]\big).
		\]
		All coefficients and both boundary derivatives on the right are determined. This proves the claim.
	\end{proof}

	From this point onward, our goal is to solve the inverse problem of recovering the metric $g$ and the vector field $X_g$ from the DN map $\Lambda'_g$ associated with the linearized equation~\eqref{DN_linearized}. To facilitate this analysis, we introduce the following notation.
	
	Let $J=(J^1,J^2):\overline\Omega\to\overline\Omega$ be a $C^1$ diffeomorphism. Let $g = (g_{ab})_{1 \leq a, b \leq 2}$ denote a $2 \times 2$ matrix-valued function representing a Riemannian metric on $ \Omega $ (not necessarily a Hessian), and let $X = (X^1, X^2)$ be a smooth vector field. Under the coordinate transformation $J$, the pullbacks of the metric, vector field, and function $v$ are defined as follows:
	\begin{equation}\label{formula in change of coordinates}
		J^*g=(DJ)^T(g\circ J)DJ,\quad
		J^*X=(DJ)^{-1}(X\circ J),\quad
		J^*v=v\circ J.
	\end{equation}
	The inverse Jacobian is evaluated at the original point: $D(J^{-1})|_{J(x)}=(DJ(x))^{-1}$. With these notations in place, we now proceed to analyze the first linearized equation.

	\subsection{Determination up to a diffeomorphism and a conformal factor}
	
	We first consider the metric and drift as independent coefficients.

	\begin{lemma}[Simultaneous recovery]\label{lemma: first linearization det}
		Let $\Omega \subset \R^2$ be a bounded open simply connected domain with $C^\infty$-smooth boundary $\p\Omega$. Assume that $g_1,g_2$ are smooth positive definite symmetric matrix fields on $\overline\Omega$, that $X_1,X_2$ are smooth real vector fields, and that $\partial^\beta g_1|_{\partial\Omega}=\partial^\beta g_2|_{\partial\Omega}$ for $|\beta|\leq1$. The DN maps use the conormal density in \eqref{DN_linearized}. Let $\Lambda'_{g_j,X_j}$ be the DN map of
		\begin{equation}\label{first linearized equation in pf}
			\begin{cases}
				\big( -\Delta_{g_j}+ X_{j}\cdot \nabla \big) v_j=0  &\text{ in }\Omega,\\
				v_j=\phi &\text{ on }\p \Omega,
			\end{cases}
		\end{equation}
		where $X_j$ is a vector field, for $j=1,2$. Suppose that 
		\begin{equation}\label{eq: DN map agree first linear}
			\Lambda'_{g_1,X_1}(\phi)=\Lambda'_{g_2,X_2} (\phi), \text{ for any }\phi \in C^{\infty}(\p \Omega),
		\end{equation}
		then there exists a diffeomorphism $J: \overline{\Omega}\to \overline{\Omega}$ with $J|_{\p \Omega}=\mathrm{Id}$, and a conformal factor $c>0$ with $c|_{\p\Omega}=1$ such that 
		\begin{enumerate}[\rm(i)]
			\item \label{item 1 in first linearization}
			\begin{equation}\label{eq: unique up to iso in first linearized}
				\begin{split}
					g_1 =cJ^\ast g_2  \quad \text{and} \quad X_{1}=c^{-1}J^\ast X_{2} \quad \text{in}\quad \Omega.
				\end{split}
			\end{equation}

			\item \label{item 2 in first linearization} Moreover, if $v_j$, $j=1,2$, are solutions to \eqref{first linearized equation in pf}, there holds
			\[
			v_1=J^*v_2.
			\]
		\end{enumerate}
		Here, all notations in \eqref{eq: unique up to iso in first linearized} are given in \eqref{formula in change of coordinates}.
	\end{lemma}

	\begin{remark}
		Note that the vector field $X_j$ in the above lemma could be independent of the metric $g_j$, for $j=1,2$. Hence, we do not use the notation $X_{g_j}$ given by \eqref{vector field X^b} to denote the vector field in \eqref{first linearized equation in pf} for $j=1,2$.
	\end{remark}

\begin{proof}[Proof of Lemma \ref{lemma: first linearization det}]
	For \ref{item 1 in first linearization}, the Dirichlet problems are uniquely solvable. Their conormal DN maps extend continuously from $H^{1/2}(\partial\Omega)$ to $H^{-1/2}(\partial\Omega)$, so equality on smooth boundary data gives equality of the Cauchy data sets by density. The boundary matching assumptions then allow us to apply \cite[Theorem 1.1]{IUY2012} to the two operators with their signs reversed. Its metric conclusion gives a boundary fixing diffeomorphism $J:\overline\Omega\to\overline\Omega$ and a positive function $c$ such that $g_1=cJ^*g_2$ in $\Omega$. Initially, $J,c$ have the finite regularity supplied by that theorem; their smoothness will follow below.
	
	We first check their boundary normalization. Since $J$ fixes the boundary pointwise, $DJ\tau=\tau$ for every tangent vector $\tau$. Evaluating $g_1=cJ^*g_2$ on $\tau$ and using $g_1=g_2$ on $\partial\Omega$ gives $c=1$ there. Hence $DJ$ is an isometry of the common boundary metric that fixes the tangent direction. It must also fix the metric unit normal, since the other choice would map the interior side to the exterior side. Thus, $DJ=I_2$ on $\partial\Omega$.
	
	For a common boundary value $\phi$, let $v_j$ solve $(-\Delta_{g_j}+X_j\cdot\nabla)v_j=0$ with $v_j|_{\partial\Omega}=\phi$, and put $\widetilde v_2=v_2\circ J$. In dimension two, $\Delta_{cJ^*g_2}=c^{-1}\Delta_{J^*g_2}$. The covariance of the Laplace--Beltrami operator and the chain rule therefore give
	\[
	\begin{aligned}
		\Delta_{g_1}\widetilde v_2=c^{-1}(\Delta_{g_2}v_2)\circ J
		=c^{-1}(X_2\cdot\nabla v_2)\circ J=c^{-1}(J^*X_2)\cdot\nabla\widetilde v_2.
	\end{aligned}
	\]
	Here $J^*X_2=(DJ)^{-1}(X_2\circ J)$, as in \eqref{formula in change of coordinates}.
	
	The boundary value of $\widetilde v_2$ is $\phi$. Moreover, $J=\mathrm{Id}$ and $DJ=I_2$ on $\partial\Omega$ imply $\nabla\widetilde v_2=\nabla v_2$ there. Since the boundary metrics agree, the conormal densities satisfy
	\[
	\partial_{\nu_{g_1}}\widetilde v_2
	=\partial_{\nu_{g_2}}v_2
	=\partial_{\nu_{g_1}}v_1
	\quad\text{on }\partial\Omega,
	\]
	where the last equality uses \eqref{eq: DN map agree first linear}. Composition with $J^{-1}$ gives the converse correspondence of solutions. Consequently, the two operators
	\[
	L_1=-\Delta_{g_1}+X_1\cdot\nabla,\quad
	\widetilde L_2=-\Delta_{g_1}+Y\cdot\nabla,\quad
	Y=c^{-1}J^*X_2
	\]
	have the same DN map with respect to the fixed metric $g_1$. Thus the remaining inverse problem is to determine the drift, not the principal part.
	
	This fixed metric problem is treated in \cite[Lemma 4.2]{nurminen2023inverse}; see also \cite{tzou2017reflection,guillarmou2011identification} for related magnetic inverse problems. The boundary agreement of the drift fields required there follows from \cite[Theorem 1.3]{IUY2012}: use a common global isothermal coordinate system for $g_1$ and normalize both equations to have principal part $\Delta$. The boundary traces and normal derivatives are transformed by the same composition and the same nonzero boundary factor, so their Cauchy data remain equal. The boundary coefficient determination in that theorem gives equality of the transformed drifts and hence $X_1=Y$ on $\partial\Omega$.
	
	We also specify why the scalar gauge is trivial in the present full boundary setting. In the same isothermal coordinates, \cite[Corollary 1.1]{IUY2012} gives the multiplication gauge for the two normalized operators. Pulling this relation back gives a function $\eta$, with $\eta|_{\partial\Omega}=0$, such that
	\begin{equation}\label{eq: operator-gauges}
		L_1=e^{-\eta}\widetilde L_2e^\eta.
	\end{equation}
	The common scalar normalization commutes with multiplication by $e^\eta$. Thus, \eqref{eq: operator-gauges} is obtained from the inverse boundary result, rather than from a comparison of the displayed differential expressions.
	
	Both operators have real coefficients and no zeroth order term. Applying \eqref{eq: operator-gauges} to the constant function gives $\widetilde L_2(e^\eta)=0$. Since $\widetilde L_2(1)=0$ and $e^\eta=1$ on $\partial\Omega$, Dirichlet uniqueness gives $e^\eta=1$ throughout $\Omega$. If $\eta$ is complex, apply the maximum principle separately to the real and imaginary parts of $e^\eta-1$. Hence, $L_1=\widetilde L_2$, so $(X_1-Y)\cdot\nabla f=0$ for every smooth $f$. Taking $f(x)=x_k$ for $k=1,2$ yields
	\begin{equation}\label{eq:no_gauge_for_X}
		X_1=c^{-1}J^*X_2\quad\text{in }\Omega.
	\end{equation}
	Together with the metric relation, this proves \eqref{eq: unique up to iso in first linearized}.
	
	For \ref{item 2 in first linearization}, $v_1$ and $\widetilde v_2$ now solve the same uniformly elliptic equation with no zeroth order term and have the same Dirichlet value. Uniqueness gives $v_1=\widetilde v_2=J^*v_2$. Finally, the established operator relation and the chain rule give
	\[
	L_1J^m=c^{-1}(X_2^m-X_{g_2}^m)\circ J,\quad
	J^m|_{\partial\Omega}=x_m,\quad m=1,2,
	\]
	where $X_{g_2}^m=\Delta_{g_2}x_m$. Taking determinants in $g_1=c(DJ)^T(g_2\circ J)DJ$, we can express $c$ as $\sqrt{\det g_1}/(\det DJ\sqrt{\det(g_2\circ J)})$. The denominator is bounded away from zero. Hence, $J\in C^{k,\alpha}$ implies $c\in C^{k-1,\alpha}$, and the global Schauder estimate for the displayed Dirichlet problem improves $J$ to $C^{k+1,\alpha}$. Starting with the regularity furnished by \cite[Theorem 1.1]{IUY2012} and iterating gives $J,c\in C^\infty(\overline\Omega)$.
\end{proof}

	We mention that a more general version of Lemma \ref{lemma: first linearization det} on Riemannian surfaces, based on the proof in \cite{CLT24}, will appear in a work by the first-mentioned author.

	\subsection{Determination of the diffeomorphism via Christoffel symbols}\label{sec: isometry Christoffel}
	We prove that $J=\mathrm{Id}$ on $\overline\Omega$. Applying Lemma~\ref{lemma: first linearization det} with $X_j=X_{g_j}$ gives a boundary fixing diffeomorphism $J$ and a positive function $c$ satisfying
	\begin{equation}\label{eq:condition}
		g_1 =cJ^\ast g_2  \quad \text{and} \quad X_{g_1}=c^{-1}J^\ast X_{g_2} \quad \text{in}\quad \Omega,
	\end{equation}
	where $X_{g_j}$ is the vector field given by \eqref{vector field X^b} with components $X_{g_j}^i$, for $i,j=1,2$. 
	In addition, 
	let $\Gamma(g)_{kl}^i$ be the Christoffel symbol associated with the metric $g$, which is given by 
	\begin{equation}\label{Christoffel symbol}
		\Gamma(g)_{kl}^i := \frac{1}{2}g^{im}\Big( \frac{\p g_{mk}}{\p x^l}+ \frac{\p g_{ml}}{\p x^k} -\frac{\p g_{kl}}{\p x^m}\Big), \quad \text{for }1\leq i,k,l\leq 2,
	\end{equation}
	for a given metric $g$. We also note that
	\begin{equation}\label{eq:contracted_christoffel}
		\begin{split}
			X_{g}^i=-g^{kl}\Gamma(g)_{kl}^i,
		\end{split}
	\end{equation}
	by standard formulas in Riemannian geometry, where $X_g$ is given by \eqref{vector field X^b}.

Under the conformal scaling $\widetilde g=e^{2\sigma}g$, the Christoffel symbols satisfy
\[
\Gamma(\widetilde g)^i_{kl}=\Gamma(g)^i_{kl}+\delta_k^i\partial_l\sigma+\delta_l^i\partial_k\sigma-g_{kl}g^{ij}\partial_j\sigma.
\]
Taking $\sigma=\frac12\log c$ gives
	\begin{equation}\label{change of variable in Christoffel}
		\begin{split}
			\Gamma^i_{kl}(g_1)=\Gamma^i_{kl}(cJ^\ast g_2)=\Gamma^{i}_{kl}(J^\ast g_2) + \frac{1}{2} \big( \delta^{i}_{k} \partial_{l} \log c + \delta^{i}_{l} \partial_{k} \log c - (g_1)_{kl} g_1^{ij}\partial_{j} \log c \big).
		\end{split}
	\end{equation}	
	Here, $\delta^i_k=\begin{cases}
		1 &\text{for }i=k\\
		0 &\text{otherwise}
	\end{cases}$ denotes the Kronecker delta.

	Using Lemma \ref{lemma: first linearization det}, we can prove Theorem \ref{theorem: J=Id in intro}. For readers' convenience, let us recall Theorem \ref{theorem: J=Id in intro} as follows. 
	
	\begin{theorem}\label{theorem: J=Id}
		Let $\Omega \subset \R^2$ be a bounded open simply connected domain with $C^\infty$-smooth boundary $\p\Omega$. Assume that $g_1,g_2$ are smooth positive definite symmetric matrix fields and that their boundary derivatives through order one agree. Let $\Lambda'_{g_j}$ be the DN map of 
		\begin{equation}\label{non-divergence elliptic equ}
			\begin{cases}
				g^{ab}_j\p_{ab} v_j =0 &\text{ in }\Omega,\\
				v_j = \phi &\text{ on }\p \Omega,
			\end{cases}
		\end{equation}
		for $j=1,2$. Suppose that 
		\begin{equation}
			\Lambda'_{g_1}(\phi)=\Lambda'_{g_2}(\phi) \quad \text{for any}\quad \phi \in C^\infty(\p \Omega),
		\end{equation}
		then there exists a conformal factor $c>0$ with $c|_{\p \Omega}=1$, such that
		\begin{equation}\label{g_1 = c g_2}
			g_1=cg_2 \text{ in }\Omega.
		\end{equation}
	\end{theorem}
	
	\begin{remark}
		Uniform convexity is not used in Theorem~\ref{theorem: J=Id}. The boundary matching and topological assumptions are those under which we apply Lemma~\ref{lemma: first linearization det}. Once the relations \eqref{eq: unique up to iso in first linearized in the pf} are available, the argument proving $J=\mathrm{Id}$ uses the two-dimensional contraction of the Christoffel transformation and uniqueness for a uniformly elliptic Dirichlet problem.
	\end{remark}

	\begin{proof}[Proof of Theorem \ref{theorem: J=Id}]
		Let us divide the proof into several steps:\\
		
		{\it Step 1. Initialization.}\\

		\noindent	First, we rewrite \eqref{non-divergence elliptic equ} in the form  
		\begin{equation}\label{non-divergence elliptic equ 2}
			\begin{cases}
				\big(-\Delta_{g_j} + X_{g_j}\cdot\nabla\big) v_j = 0 & \text{in } \Omega,\\
				v_j = \phi & \text{on } \partial\Omega,
			\end{cases}
		\end{equation}
		where $X_{g_j}$ is the vector field given in \eqref{vector field X^b} with $g = g_j$, for $j=1,2$.  
		By Lemma~\ref{lemma: first linearization det}~\ref{item 1 in first linearization}, there exists a diffeomorphism  
		$J: \overline{\Omega} \to \overline{\Omega}$ satisfying $J|_{\partial\Omega} = \mathrm{Id}$ and a conformal factor $c>0$ with $c|_{\partial\Omega}=1$, such that  
		\begin{equation}\label{eq: unique up to iso in first linearized in the pf}
			g_1 = c\, J^\ast g_2, 
			\quad 
			X_{g_1} = c^{-1} J^\ast X_{g_2} 
			\quad \text{in } \Omega.
		\end{equation}

		The idea is that although the metric is determined only up to pullback by a diffeomorphism and a conformal factor, the specific structure of the vector field $X_g$ (which encodes contracted Christoffel symbols) imposes additional compatibility conditions that force the diffeomorphism to be trivial.

		\bigskip
		
		{\it Step 2. Unique determination of the diffeomorphism.}\\
		
		\noindent	We next claim 
		\begin{equation}\label{claim: J=Id}
			J= \mathrm{Id} \text{ in }\overline{\Omega},
		\end{equation}
		or $J(x)=x$ for all $x\in \overline{\Omega}$. To this end, let us write $\wt x = J(x)$ and use the typical convention to denote by $\frac{\p x^i}{\p \wt x^m}$ the components of the differential of the inverse of $J$ (evaluated at $J$). We have the standard Christoffel symbols transform as
		\begin{equation}\label{de_dphi_aux1}
			\begin{split}
				\Gamma^i_{kl}(J^\ast g_2) =
				\frac{\partial  x^i}{\partial \wt x^m}\,
				\frac{\partial  \wt x^a}{\partial  x^k}\,
				\frac{\partial \wt x^b}{\partial  x^l}\,
				\Gamma^m_{ab}(g_2) \circ J
				+ 
				\frac{\partial^2  \wt x^m}{\partial  x^k \partial x^l}\,
				\frac{\partial  x^i}{\partial \wt x^m}.
			\end{split}
		\end{equation}
		Multiplying \eqref{de_dphi_aux1} by the matrix $(J^\ast g_2)^{kl}$, and applying \eqref{eq:contracted_christoffel}, we can obtain 
		\begin{equation}\label{de_dphi_aux2}
			X^i_{J^*g_2}
			=\frac{\partial x^i}{\partial\widetilde x^m}(X^m_{g_2}\circ J)
			-(J^*g_2)^{kl}\frac{\partial^2\widetilde x^m}{\partial x^k\partial x^l}\frac{\partial x^i}{\partial\widetilde x^m}.
		\end{equation}
		The minus sign in the second term comes from $X_g^i=-g^{kl}\Gamma(g)^i_{kl}$. The Christoffel transformation in \eqref{de_dphi_aux1} itself has a plus sign.

		By \eqref{change of variable in Christoffel} and \eqref{de_dphi_aux1}, one can find 
		\begin{equation}
			\begin{split}
				\Gamma^i_{kl}(g_1)=\frac{\partial  x^i}{\partial \wt x^m}\,
				\frac{\partial  \wt x^a}{\partial  x^k}\,
				\frac{\partial \wt x^b}{\partial  x^l}\,
				\Gamma^m_{ab}(g_2) \circ J
				+ 
				\frac{\partial^2  \wt x^m}{\partial x^k \partial x^l}\,
				\frac{\partial  x^i}{\partial \wt x^m}  + \frac 12 \big( \delta^{i}_{k} \partial_{l} \log c +  \delta^{i}_{l} \partial_{k} \log c -  (g_1)_{kl} g_1^{ij}\partial_{j} \log c\big) .
			\end{split}
		\end{equation}
		By \eqref{eq:contracted_christoffel}, \eqref{change of variable in Christoffel} and \eqref{de_dphi_aux2}, one has 
		\begin{equation}\label{J*g_2 Gamma 1}
			\begin{aligned}
				(J^*g_2)^{kl}\Gamma^i_{kl}(g_1)
				={}&-X^i_{J^*g_2}+\frac12(J^*g_2)^{kl}\big(\delta_k^i\partial_l\log c+\delta_l^i\partial_k\log c-(g_1)_{kl}g_1^{ij}\partial_j\log c\big)\\
				={}&-\frac{\partial x^i}{\partial\widetilde x^m}(X^m_{g_2}\circ J)
				+(J^*g_2)^{kl}\frac{\partial^2\widetilde x^m}{\partial x^k\partial x^l}\frac{\partial x^i}{\partial\widetilde x^m}\\
				&+\frac12(J^*g_2)^{kl}\big(\delta_k^i\partial_l\log c+\delta_l^i\partial_k\log c-(g_1)_{kl}g_1^{ij}\partial_j\log c\big).
			\end{aligned}
		\end{equation}
		Via \eqref{eq: unique up to iso in first linearized} and \eqref{eq:contracted_christoffel}, the left-hand side of \eqref{J*g_2 Gamma 1} is 
		\begin{equation}\label{J*g_2 Gamma 2}
			(J^*g_2)^{kl}\Gamma^i_{kl}(g_1)=cg_1^{kl}\Gamma^i_{kl}(g_1)=-cX_{g_1}^i.
		\end{equation}

		On the one hand, plugging \eqref{J*g_2 Gamma 2} into \eqref{J*g_2 Gamma 1}, and using the second relation in \eqref{eq: unique up to iso in first linearized in the pf}, i.e., $X_{g_1}=c^{-1}J^\ast X_{g_2}$, we obtain
		\begin{equation}\label{J*g_2 Gamma 3}
			\begin{aligned}
				(J^*X_{g_2})^i=cX_{g_1}^i
				={}&\frac{\partial x^i}{\partial\widetilde x^m}(X^m_{g_2}\circ J)
				-(J^*g_2)^{kl}\frac{\partial^2\widetilde x^m}{\partial x^k\partial x^l}\frac{\partial x^i}{\partial\widetilde x^m}\\
				&-\frac12(J^*g_2)^{kl}\big(\delta_k^i\partial_l\log c+\delta_l^i\partial_k\log c-(g_1)_{kl}g_1^{ij}\partial_j\log c\big).
			\end{aligned}
		\end{equation}
		On the other hand, via \eqref{formula in change of coordinates}, we also have 
		\begin{equation}\label{J*g_2 Gamma 4}
			\begin{split}
				\left(J^*X_{g_2}\right)^i=\frac{\partial x^i}{\partial \wt x^m}\, X_{g_2}^m \circ J,
			\end{split}
		\end{equation}
		so we can insert \eqref{J*g_2 Gamma 4} into the left-hand side of \eqref{J*g_2 Gamma 3}, which can be canceled by the first term in the right-hand side of \eqref{J*g_2 Gamma 3}. Thus, one can obtain 
		\begin{equation}\label{J*g_2 Gamma 5}
			(J^*g_2)^{kl}\frac{\partial^2\widetilde x^m}{\partial x^k\partial x^l}\frac{\partial x^i}{\partial\widetilde x^m}
			=-\frac12(J^*g_2)^{kl}\big(\delta_k^i\partial_l\log c+\delta_l^i\partial_k\log c-(g_1)_{kl}g_1^{ij}\partial_j\log c\big).
		\end{equation}
		Finally, using $g_1 =cJ^\ast g_2$ to the both sides of \eqref{J*g_2 Gamma 5}, we have
		\begin{equation}\label{J*g_2 Gamma 6}
			g_1^{kl}\frac{\partial^2\widetilde x^m}{\partial x^k\partial x^l}\frac{\partial x^i}{\partial\widetilde x^m}
			=-\frac12g_1^{kl}\big(\delta_k^i\partial_l\log c+\delta_l^i\partial_k\log c-(g_1)_{kl}g_1^{ij}\partial_j\log c\big).
		\end{equation}
		which can also be written in an equivalent form 
		\begin{equation}\label{J*g_2 Gamma 7}
			g_1^{kl}\frac{\partial^2\widetilde x^m}{\partial x^k\partial x^l}
			=-\frac12\frac{\partial\widetilde x^m}{\partial x^i}
			\big(2g_1^{ij}\partial_j\log c-2g_1^{ij}\partial_j\log c\big)=0.
		\end{equation}
		where we used $g_1^{kl}(g_1)_{kl}=\tr ( \mathrm{I}_2)=2$ for the last term in the right-hand side of \eqref{J*g_2 Gamma 6}.
		
		Thus, for the individual $m=1,2$, the equation \eqref{J*g_2 Gamma 7} implies that 
		\begin{equation}\label{J*g_2 Gamma 8}
			\begin{cases}
				g_{1}^{kl}\frac{\partial^2 \wt x^m}{\partial x^k \partial x^l}=0 &\text{ in }\Omega, \\
				\wt x^m = x^m &\text{ on }\p \Omega.
			\end{cases}
		\end{equation}
	    For each $m=1,2$, the coordinate function $x^m$ satisfies this same Dirichlet problem. Hence, $J^m-x^m$ solves a homogeneous uniformly elliptic equation with zero boundary value. Dirichlet uniqueness gives $J^m=x^m$ in $\Omega$, proving \eqref{claim: J=Id}. \\
		
		{\it Step 3. Summary.}\\
		
		\noindent Finally, using \eqref{eq: unique up to iso in first linearized in the pf}, we know that 
		\begin{equation}
			g_1 =\underbrace{cJ^\ast g_2 = cg_2}_{\text{By $J=\mathrm{Id}$}} \text{ in }\Omega,
		\end{equation}
		which proves the assertion \eqref{g_1 = c g_2}.
	\end{proof}

	\begin{remark}~
		\begin{enumerate}[\rm(i)]
			\item As shown in \eqref{J*g_2 Gamma 7}, the mapping $J$ satisfies the elliptic equation \eqref{J*g_2 Gamma 8} independently of the conformal factor $c>0$.
			\item In Theorem \ref{theorem: J=Id}, uniform convexity is not used here. Simple connectedness is assumed in the planar inverse result used in Lemma~\ref{lemma: first linearization det}. The contraction argument proving $J=\mathrm{Id}$ does not itself use this topological assumption.
		\end{enumerate}
	\end{remark}
	
	The remainder of the paper is devoted to recovering the conformal factor $c>0$ in $\Omega$ using suitable CGO solutions for the first linearized equation.

	\section{Complex geometrical optics solutions}\label{sec: CGOs}

	This section is devoted to constructing CGO solutions for the first linearized equation~\eqref{Linearized main equation} and its adjoint equation \eqref{adjoint of 1st linearization}. We also derive expansion formulas for the correction terms and provide estimates for the related oscillatory integrals.

\subsection{Isothermal coordinates}\label{sec: isothermal coordinates}

Choose an orientation preserving isothermal coordinate map $\chi:\overline{\widetilde\Omega}\to\overline\Omega$ such that
\begin{equation}\label{isothermal for g_1}
	\chi^*g_1=\mu I_2,
\end{equation}
where $\mu>0$ is smooth; see \cite{Ahlfors1966}. We fix this direction for $\chi$ throughout the paper. Thus, $\mathbf v=v\circ\chi$ is a function on $\widetilde\Omega$. Covariance of the Laplace--Beltrami operator gives
\begin{equation}\label{equ: isothermal 1}
	-\mu^{-1}\Delta\mathbf v+(\chi^*X_{g_1})\cdot\nabla\mathbf v=0.
\end{equation}
Multiplication by $\mu$ gives
\begin{equation}\label{equ: isothermal 2}
	L\mathbf v=0,\quad L=-\Delta+X\cdot\nabla,\quad X=\mu\chi^*X_{g_1}.
\end{equation}
The boundary value becomes $\phi\circ\chi$.

The transpose of $L$ is $L^t=-\Delta-X\cdot\nabla-\nabla\cdot X$. Indeed, integration by parts gives $\int kX\cdot\nabla f\,dA=-\int f\nabla\cdot(Xk)\,dA$ for compactly supported functions. Since $\chi^*dV_{g_1}=\mu\,dA$, the pullback of the transpose of $\mathcal L_{g_1}$ with respect to $dV_{g_1}$ is $\mu^{-1}L^t$. The adjoint solution therefore satisfies
\begin{equation}\label{adjoint first linear eq in isothermal}
	\Delta\mathbf v^*+\nabla\cdot(X\mathbf v^*)=0.
\end{equation}
The star on a solution denotes this transpose equation. The integral identities use bilinear pairings without complex conjugation.

\subsection{The construction of CGOs}
	
We first give the construction on a Riemann surface, following the right-inverse setting of \cite{guillarmou2011calderon,guillarmou2011identification}, and then specify the scalar operators in the global coordinate used in this paper.
	
Let $\Sigma\Subset M_0\Subset M\Subset\widetilde M$ be smooth relatively compact surface domains, with a smooth Riemannian metric on a neighborhood of $\overline M$. Consider
\begin{equation}\label{equ of CGO solutions}
	L=-\Delta_g+X\cdot\nabla+q,
\end{equation}
where $X$ is a smooth real vector field and $q$ is a smooth complex-valued function. Extend the coefficients smoothly to the larger domains. All coefficient cutoffs introduced below equal one near $\overline\Sigma$. In the planar application, $\Sigma$ contains the closure of the domain denoted by $U$ in the integral identity.
	
For the original and transpose equations in isothermal coordinates, the operators are
\begin{align}
	-\Delta+X\cdot\nabla,\label{linear general}\\
	-\Delta-X\cdot\nabla-\nabla\cdot X.\label{linear adjoint general}
\end{align}
We seek solutions of \eqref{equ of CGO solutions} of the form
\begin{equation}\label{CGO solutions}
	v=F_A^{-1}e^{\Phi/h}(a+r_h),
\end{equation}
where $\Phi$ and $a$ are holomorphic, $F_A$ is smooth and nonzero, and $r_h$ is a correction term. We also construct antiholomorphic solutions. The surface construction and the noncritical expansions are stated independently of the later choice of three phases.

\subsubsection{Calculus on Riemannian surfaces}

The complexified cotangent bundle splits as $\C T^*M=T^*_{1,0}M\oplus T^*_{0,1}M$, the eigenspaces of the Hodge star with eigenvalues $-\mathsf{i}$ and $\mathsf{i}$, respectively. In a holomorphic coordinate $z=x_1+\mathsf{i}x_2$, these bundles are spanned by $dz$ and $d\bar z$. The invariant operators are $\partial=\pi_{1,0}d$ and $\overline\partial=\pi_{0,1}d$ on functions. Hence, $\partial f=(\partial_zf)dz$ and $\overline\partial f=(\partial_{\bar z}f)d\bar z$, with the scalar derivatives in \eqref{definition of d and d-bar operators}. For a one-form $\omega=\omega_1dx_1+\omega_2dx_2$,
\begin{equation}\label{eq:x_formula}
	\omega=\frac12(\omega_1-\mathsf{i}\omega_2)dz+\frac12(\omega_1+\mathsf{i}\omega_2)d\bar z=\omega_{1,0}+\omega_{0,1}.
\end{equation}
On one-forms, $\partial=d\pi_{0,1}$ and $\overline\partial=d\pi_{1,0}$. Thus $\partial(u\,dz+v\,d\bar z)=(\partial_zv)dz\wedge d\bar z$ and $\overline\partial(u\,dz+v\,d\bar z)=(\partial_{\bar z}u)d\bar z\wedge dz$.

Write the metric locally as $g=\rho|dz|^2$. Then
\begin{equation}\label{dV_g}
	dV_g=\rho\,dx_1\wedge dx_2=\frac{\mathsf{i}\rho}{2}dz\wedge d\bar z,\quad
	\star(dz\wedge d\bar z)=-2\mathsf{i}\rho^{-1}.
\end{equation}
The Hermitian adjoint on forms is $\overline\partial^*=-\mathsf{i}\star\partial$. In scalar coordinates,
\[
\overline\partial^*(f\,d\bar z)=-2\rho^{-1}\partial_zf,\quad
2\overline\partial^*\overline\partial=-\Delta_g.
\]
In particular, the adjoint is not simply $\partial_z$.

Choose right inverses $R=\overline\partial^{-1}$ and $S=\overline\partial^{*-1}$ on $M$ as in \cite[Proposition 2.1 and Lemma 2.1]{guillarmou2011identification}. They satisfy
\begin{equation}\label{eq:inverse_of_op}
	\overline\partial R\omega=\omega,\quad
	\overline\partial^*Sf=f
\end{equation}
for compactly supported smooth inputs. On relatively compact subdomains, they map $W^{k,p}$ continuously to $W^{k+1,p}$ for $1<p<\infty$. Their local representations are Cauchy parametrices plus operators with smooth kernels. We choose complex-linear right inverses; taking the complex-linear part of a right inverse, if necessary, preserves these identities and mapping properties.

We will use the corresponding left compositions only on compactly supported inputs. The local parametrices give $R\overline\partial u=u+K_0u$ and $S\overline\partial^*\omega=\omega+K_1\omega$, where the kernels of $K_0,K_1$ are smooth when input and output are restricted to fixed compact subsets of $M$. For $R$, this also follows from the smooth kernel projection in the right-inverse construction. For $S$, it follows by composing its local elliptic parametrix with $\overline\partial^*$; the diagonal singularity gives the identity, and the remaining kernel is smooth. These errors will be included in the noncritical expansion.

Let $\Phi$ be holomorphic near $\overline M$, put $\psi=\operatorname{Im}\Phi$, and set $E_h=e^{(\Phi-\overline\Phi)/h}$. With a fixed extension operator $\mathcal E$ and restriction $\mathcal R$ to $M_0$, define
\begin{equation}\label{eq:p_psi_inv}
	R_\psi=\mathcal R R E_h^{-1}\mathcal E,\quad
	S_\psi=\mathcal R S E_h\mathcal E.
\end{equation}
In the integral equations below, all inputs are multiplied by smooth coefficients compactly supported in $M_0$. For those inputs, $\mathcal E$ is extension by zero. Therefore, differentiation on $M_0$ gives the exact right-inverse identities. We do not commute an arbitrary extension operator with the oscillatory factor.

For a Morse function $\psi$, the estimates in \cite[Lemmas 2.2 and 2.3]{guillarmou2011identification} give, for $p>2$ and compactly supported $f$,
\begin{equation}\label{eq:sobo_decay}
	\|R_\psi f\|_{L^q}+\|S_\psi f\|_{L^q}
	\leq C\|f\|_{W^{1,p}}
	\begin{cases}
		h^{2/3},&1<q<2,\\
		h^{1/q},&2\leq q\leq p.
	\end{cases}
\end{equation}
Each term is interpreted in its appropriate function or one-form space. The norms and supports are on fixed domains. Choosing $q_0\in(3/2,2)$ and $q_1\in(2,p]$, and interpolating with $1/2=(1-\theta)/q_0+\theta/q_1$, gives the exponent
\[
(1-\theta)\frac23+\frac{\theta}{q_1}=\frac12+(1-\theta)\Big(\frac23-\frac1{q_0}\Big)>\frac12.
\]
Consequently, for some $\epsilon>0$,
\begin{equation}\label{eq:sobo_decayL2}
	\|R_\psi f\|_{L^2}+\|S_\psi f\|_{L^2}\leq Ch^{1/2+\epsilon}\|f\|_{W^{1,p}}.
\end{equation}
Replacing $\Phi$ by $-\Phi$ gives the opposite signs of oscillation.

We next give the noncritical expansion separately. Assume $d\Phi\neq0$ on $\overline M_0$, and take the input supported in a fixed compact subset of $M_0$. For a $(0,1)$-form $f$, define scalar coefficients
\begin{equation}\label{eq: noncritical-right-coefficients}
	F^1=\frac{f}{\overline\partial\overline\Phi},\quad
	F^{j+1}=-\frac{\overline\partial F^j}{\overline\partial\overline\Phi}.
\end{equation}
Division means division of two sections of the same one-dimensional bundle. These functions are smooth, compactly supported and independent of $h$.

Since $\overline\partial E_h^{-1}=h^{-1}E_h^{-1}\overline\partial\overline\Phi$,
\begin{equation}\label{eq: no critical point expansion}
	E_h^{-1}f=h\overline\partial(E_h^{-1}F^1)-hE_h^{-1}\overline\partial F^1.
\end{equation}
Applying the right inverse, including its left-composition error, gives
\begin{equation}\label{eq: no critical point expansion II}
	R_\psi f=hE_h^{-1}F^1-hR_\psi(\overline\partial F^1)+h\mathcal R K_0(E_h^{-1}F^1).
\end{equation}
The last term is $O(h^2\|f\|_{W^{1,p}})$ in $L^p$: integrate its smooth kernel once in the input variable against the noncritical phase. It follows that
\begin{equation}\label{eq: no critical point expansion III}
	\|R_\psi f\|_{L^p}\leq Ch\|f\|_{W^{1,p}},\quad 1<p<\infty.
\end{equation}

For the adjoint right inverse, define the pointwise isomorphism $\mathcal A_\Phi:T^*_{0,1}M\to\C$ by $\mathcal A_\Phi\omega=-\mathsf{i}\star(\partial\Phi\wedge\omega)$. In a local coordinate, $\mathcal A_\Phi(f\,d\bar z)=-2\rho^{-1}\Phi'f$. Hence it is invertible when $d\Phi\neq0$. For scalar $f$, put $B^1(f)=\mathcal A_\Phi^{-1}f$ and $B^{j+1}(f)=-\mathcal A_\Phi^{-1}\overline\partial^*B^j(f)$. The product rule is
\[
\overline\partial^*(E_h\omega)=E_h\overline\partial^*\omega+h^{-1}E_h\mathcal A_\Phi\omega.
\]
The same one-step argument gives
\begin{equation}\label{eq: no critical point expansion III adjoint}
	\|S_\psi f\|_{L^p}\leq Ch\|f\|_{W^{1,p}}.
\end{equation}

For every integer $N\geq1$, the full right-inverse expansions are
\begin{equation}\label{equ: d-bar inverse of f}
	\begin{aligned}
		R_\psi f&=E_h^{-1}\sum_{j=1}^Nh^jF^j-h^NR_\psi(\overline\partial F^N)+O_{W^{k,p}}(h^\infty),\\
		S_\psi f&=E_h\sum_{j=1}^Nh^jB^j(f)-h^NS_\psi(\overline\partial^*B^N(f))+O_{W^{k,p}}(h^\infty).
	\end{aligned}
\end{equation}
Here the inputs are smooth and fixed, $k$ is any fixed nonnegative integer, and $O(h^\infty)$ means $O(h^L)$ for every $L$. To verify the first identity, differentiate $E_h^{-1}\sum_{j=1}^Nh^jF^j$. Consecutive terms cancel by \eqref{eq: noncritical-right-coefficients}, leaving $E_h^{-1}f+h^NE_h^{-1}\overline\partial F^N$. Apply $R$ and use the left-composition formula. The smooth-kernel term is $O(h^\infty)$ by repeated integration by parts in its input variable, where $d\psi$ does not vanish. Output derivatives only differentiate the smooth kernel. The second identity follows from the displayed product rule for $\overline\partial^*$ and the recursion for $B^j$. This proves both expansions, including their smoothing errors.

For the standard planar Cauchy right inverses on compactly supported inputs, the left composition errors are zero, and the two identities in \eqref{equ: d-bar inverse of f} are exact without $O(h^\infty)$.

\subsubsection{CGOs with holomorphic phase}

We construct
\begin{equation}\label{CGO solution with holo phase}
	v=F_A^{-1}e^{\Phi/h}(a+r_h),
\end{equation}
where $\Phi,a$ are holomorphic. Let $\omega=X^\flat$ be the metric dual of the real vector field. Set $A=(\mathsf{i}/2)\omega_{0,1}$ and choose $\alpha$ with $\overline\partial\alpha=A$ on a neighborhood of $\overline M_0$, using the right inverse after a compactly supported extension. Define
\begin{equation}\label{equ: F_A nonzero}
	F_A=e^{\mathsf{i}\alpha}.
\end{equation}
Then $F_A$ is nonzero. Also put
\begin{equation}\label{eq:formular_for_FA}
	F_{\overline A}=e^{\mathsf{i}\overline\alpha}=\overline{F_A}^{-1}.
\end{equation}
We write $F=F_A$, $H=F_{\overline A}$ and $B=FH$ within this construction. Notice that $B=e^{\mathsf{i}(\alpha+\overline\alpha)}$; it is not $|F|^2$.

The factorization for the real drift equation is
\begin{equation}\label{eq:L_magnetic form}
	L=2H\overline\partial^*\big[H^{-1}F^{-1}\overline\partial(F\,\cdot\,)\big]+Q,
\end{equation}
where
\begin{equation}\label{eq:formula_for_Q}
	Q=q-\frac12\operatorname{div}_gX-\frac{\mathsf{i}}2\star d\omega+\frac14|X|_g^2.
\end{equation}
We verify it, including the zeroth order term. Since $\overline\partial F=\mathsf{i}AF$ and $\partial H^{-1}=-\mathsf{i}\overline A H^{-1}$,
\[
H^{-1}F^{-1}\overline\partial(Ff)=H^{-1}(\overline\partial f+\mathsf{i}Af).
\]
Using $\overline\partial^*=-\mathsf{i}\star\partial$ and the product rule for one-forms gives
\begin{equation}\label{eq: geometric-factorization-calculation}
	\begin{aligned}
		2H\overline\partial^*\big[H^{-1}F^{-1}\overline\partial(Ff)\big]
		=-2\mathsf{i}\star\big(\partial\overline\partial f-\mathsf{i}\overline A\wedge\overline\partial f-\mathsf{i}A\wedge\partial f+(\mathsf{i}\partial A+\overline A\wedge A)f\big).
	\end{aligned}
\end{equation}
In a coordinate with $g=\rho|dz|^2$, write $\omega=\omega_1dx_1+\omega_2dx_2$. Then $A=\mathsf{i}(\omega_1+\mathsf{i}\omega_2)d\bar z/4$ and $\overline A=-\mathsf{i}(\omega_1-\mathsf{i}\omega_2)dz/4$. Moreover,
\[
\overline A\wedge\overline\partial f+A\wedge\partial f
=-\frac12(\omega_1\partial_{x_1}f+\omega_2\partial_{x_2}f)\,dx_1\wedge dx_2.
\]
Its contribution in \eqref{eq: geometric-factorization-calculation} is $\rho^{-1}(\omega_1\partial_{x_1}f+\omega_2\partial_{x_2}f)=X\cdot\nabla f$.

For the zeroth order term, put $\omega_\pm=\omega_1\pm\mathsf{i}\omega_2$. The same calculation in scalar coordinates is
\[
2H\overline\partial^*\big[H^{-1}F^{-1}\overline\partial(Ff)\big]
=-4\rho^{-1}(\partial-\omega_-/4)(\overline\partial-\omega_+/4)f.
\]
Expansion gives $-\Delta_gf+X\cdot\nabla f+\rho^{-1}(\partial\omega_+-\omega_+\omega_-/4)f$. Finally,
\[
\rho^{-1}\partial\omega_+=\frac12\operatorname{div}_gX+\frac{\mathsf{i}}2\star d\omega,\quad
\rho^{-1}\omega_+\omega_-=|X|_g^2.
\]
Addition of \eqref{eq:formula_for_Q} proves \eqref{eq:L_magnetic form}.

Set $V'=-B$ and $V=Q/(2B)$. Choose a smooth cutoff supported in $M_0$ and equal to one near $\overline\Sigma$, and multiply $V,V'$ by this cutoff. We use the same letters for these extensions. Their values near $\overline\Sigma$ are unchanged, and all inputs to the right inverses below are compactly supported.

\begin{lemma}[CGO solutions]\label{lemma: CGO solution}
	Let $\Phi$ be holomorphic near $\overline M_0$, with Morse imaginary part, and let $a$ be holomorphic there. For sufficiently small $h>0$, \eqref{CGO solution with holo phase} solves $Lv=0$ near $\overline\Sigma$, where
	\begin{align}
		r_h&=-R_\psi(V's_h),\label{eq: def of rh}\\
		s_h&=-(I-T_h)^{-1}S_\psi(Va)=-\sum_{j=0}^\infty T_h^jS_\psi(Va),\label{eq: def of sh}\\
		T_h&=S_\psi V R_\psi V'.\label{def: Th}
	\end{align}
	The series converges in each fixed $L^p$, $1<p<\infty$, for sufficiently small $h$; the threshold may depend on $p$.
\end{lemma}

\begin{proof}
	For $p>2$, the inner right inverse maps $L^p$ to $W^{1,p}$ uniformly, and the outer oscillatory estimate gives $\|T_h\|_{L^p\to L^p}=O(h^{1/p})$. For $1<r<2$, both unweighted right-inverse bounds give a uniform $L^r$ norm. For each $1<p\leq2$, interpolation between a uniform bound on some $L^r$, $1<r<p$, and a decaying bound on some $L^s$, $s>2$, gives a positive power of $h$ on $L^p$. Hence, the inverses and series in \eqref{eq: def of sh} exist in all the stated spaces.
	
	Their exact equations are $\overline\partial r_h=-E_h^{-1}V's_h$ and $\overline\partial^*s_h=-E_hV(a+r_h)$. Set $u=e^{\Phi/h}(a+r_h)$ and $\zeta=e^{\overline\Phi/h}s_h$. Then $\overline\partial u=-V'\zeta$ and $\overline\partial^*\zeta=-Vu$. The second identity uses $\partial e^{\overline\Phi/h}=0$. Near $\overline\Sigma$, $-V'=B$, so $\zeta=B^{-1}\overline\partial u$. Therefore,
	\[
	Lv-Qv=2H\overline\partial^*(B^{-1}\overline\partial u)
	=2H\overline\partial^*\zeta=-2HVu=-QF^{-1}u=-Qv.
	\]
	This proves $Lv=0$. For fixed $h$, the two integral equations and elliptic regularity give smoothness of the correction terms.
\end{proof}

We record the estimates, including the derivative bounds. For $p>2$,
\begin{equation}\label{Th norm estimate}
	\|T_h\|_{L^p\to L^p}=O(h^{1/p}),\quad
	\|T_h\|_{L^2\to L^2}=O(h^{1/2-\eta})
\end{equation}
for every fixed $\eta>0$. For the second assertion, interpolate the $L^p$ estimate with a uniform $L^r$ estimate, where $p>2$ and $r<2$ are chosen close enough to two. Likewise,
\begin{equation}\label{Th norm estimate sobo}
	\|T_h\|_{W^{1,p}\to W^{1,p}}=O(h^{1/p}),\quad
	\|T_h\|_{W^{1,2}\to W^{1,2}}=O(h^{1/2-\eta}).
\end{equation}
Indeed, $\|T_hf\|_{W^{1,p}}\leq C\|R_\psi(V'f)\|_{L^p}\leq Ch^{1/p}\|f\|_{W^{1,p}}$. The $W^{1,2}$ estimate follows by interpolation with the uniform bounds below exponent two. We use interpolation on smooth bounded domains; see \cite{BL-76_interpolation}. Independently, $T_h:L^p\to W^{1,p}$ is uniformly bounded.

The forcing term in \eqref{eq: def of sh} satisfies \eqref{eq:sobo_decayL2}. Boundedness of $(I-T_h)^{-1}$ gives
\begin{equation}\label{eq:rh_L2_estim}
	\|s_h\|_{L^2}+\|r_h\|_{W^{1,2}}=O(h^{1/2+\epsilon})
\end{equation}
for some $\epsilon>0$. For every finite $p\geq2$, use the smooth forcing in an input space $W^{1,p_1}$ with $p_1>p$. Interpolation between an output exponent $q_0\in(3/2,2)$ and $p_1$ in \eqref{eq:sobo_decay} gives an exponent strictly larger than $1/p$. Consequently,
\begin{equation}\label{eq:rh_Lp_estim}
	\|s_h\|_{L^p}+\|r_h\|_{W^{1,p}}=O(h^{1/p+\epsilon_p})
\end{equation}
for some $\epsilon_p>0$. The constructions in different $L^p$ spaces agree by uniqueness of the $L^2$ fixed point. Both $\partial r_h$ and $\overline\partial r_h$ satisfy this bound.

For the second derivatives, let $H_0$ be a smooth cutoff in a coordinate chart, compactly supported in $M_0$. Product differentiation and the local Calder\'on--Zygmund estimate give
\begin{equation}\label{eq:second_derivs_of_rh}
	\begin{aligned}
		\|H_0D^2r_h\|_{L^2}
		&\leq\|D^2(H_0r_h)\|_{L^2}+2\|\nabla H_0\otimes\nabla r_h\|_{L^2}+\|r_hD^2H_0\|_{L^2}\\
		&\leq C\|\partial\overline\partial(H_0r_h)\|_{L^2}+O(h^{1/2+\epsilon})\\
		&\leq C\|H_0\partial(E_h^{-1}V's_h)\|_{L^2}+O(h^{1/2+\epsilon})\\
		&\leq C\big(h^{-1}\|s_h\|_{L^2}+\|\partial s_{h,\bar z}\|_{L^2}+\|s_h\|_{L^2}\big).
	\end{aligned}
\end{equation}
Here $s_h=s_{h,\bar z}d\bar z$ in that chart. The equation $\overline\partial^*s_h=-E_hV(a+r_h)$ and $\overline\partial^*(s_{h,\bar z}d\bar z)=-2\rho^{-1}\partial s_{h,\bar z}$ give $\|\partial s_{h,\bar z}\|_{L^2}=O(1)$. Thus, after decreasing $\epsilon<1/2$,
\begin{equation}\label{CZ estimate for r_2}
	\|H_0D^2r_h\|_{L^2}=O(h^{-1/2+\epsilon}).
\end{equation}
A finite coordinate covering gives the same local estimate on any fixed relatively compact subdomain. The constants depend on the cutoff and the coefficients, not on $h$.

\subsubsection{CGOs with antiholomorphic phase}

The antiholomorphic construction also applies when $q$ is complex. Apply the holomorphic construction to $-\Delta_g+X\cdot\nabla+\overline q$ with phase $-\Phi$ and amplitude $\overline a$, where $a$ is antiholomorphic, and then take complex conjugates. The resulting function solves the equation with potential $q$. Its leading factor is $\overline{F_A^{-1}}=F_{\overline A}$.

Let $\widetilde R=\overline R$ and $\widetilde S=\overline S$ denote the conjugate right inverses for $\partial$ and $\partial^*$, respectively. Define $\widetilde R_\psi=\mathcal R\widetilde R E_h^{-1}\mathcal E$ and $\widetilde S_\psi=\mathcal R\widetilde S E_h\mathcal E$. The sign change of the phase before conjugation is included in these definitions. Put
\[
\widetilde Q=q-\frac12\operatorname{div}_gX+\frac{\mathsf{i}}2\star d\omega+\frac14|X|_g^2,\quad
\widetilde V'=-\overline B,\quad
\widetilde V=\frac{\widetilde Q}{2\overline B}.
\]
Use the conjugate cutoffs from the holomorphic construction. Then
\begin{equation}\label{eq:formula_for_tilde_rh}
	\widetilde v=F_{\overline A}e^{-\overline\Phi/h}(a+\widetilde r_h),\quad
	\widetilde T_h=\widetilde S_\psi\widetilde V\widetilde R_\psi\widetilde V',
\end{equation}
where
\begin{equation}\label{eq: form of tilderh}
	\widetilde s_h=-\sum_{j=0}^\infty\widetilde T_h^j\widetilde S_\psi(\widetilde Va),\quad
	\widetilde r_h=-\widetilde R_\psi(\widetilde V'\widetilde s_h).
\end{equation}
These are the conjugates of the exact holomorphic equations for $\overline q$ and phase $-\Phi$. Hence, they solve the required equation, not merely its leading transport equation. The estimates \eqref{Th norm estimate}--\eqref{CZ estimate for r_2} apply to $\widetilde T_h,\widetilde s_h,\widetilde r_h$ by conjugation. When $q$ is real, this reduces to conjugating solutions of the same operator with the opposite phase.					
					
\subsubsection{Estimates and expansions for CGOs in the absence of critical points}

Assume $d\Phi\neq0$ on $\overline M_0$. Replacing the Morse estimate by \eqref{eq: no critical point expansion III} and \eqref{eq: no critical point expansion III adjoint} in the construction gives $\|T_h\|_{L^p\to L^p}=O(h)$ and $\|T_h\|_{W^{1,p}\to W^{1,p}}=O(h)$ for every $1<p<\infty$. The same is true for $\widetilde T_h$. Therefore,
\begin{equation}\label{eq: no critical point estimate}
	\|r_h\|_{W^{1,p}}+\|s_h\|_{L^p}+\|\widetilde r_h\|_{W^{1,p}}+\|\widetilde s_h\|_{L^p}\leq C_ph.
\end{equation}
For exponents below two, agreement with the $L^2$ construction follows from uniqueness of the fixed point in $L^p$ on the bounded domain.

\begin{lemma}\label{lem: asymptotic of remainders with linear}
	For every $N\geq1$ and $1<p<\infty$, the same correction terms satisfy
	\begin{equation}\label{eq:r_h_expansion}
		r_h=\sum_{j=1}^Nh^ja_j+O_{W^{2,p}}(h^{N+1}),\quad
		\widetilde r_h=\sum_{j=1}^Nh^j\widetilde a_j+O_{W^{2,p}}(h^{N+1}).
	\end{equation}
	The coefficients are smooth, independent of $h$, and consistent as $N$ varies.
\end{lemma}

\begin{proof}[Proof of Lemma \ref{lem: asymptotic of remainders with linear}]
	We prove the holomorphic expansion by truncating and expanding the Neumann series. The weighted outer inverse has a derivative loss:
	\begin{equation}\label{eq: weighted-two-derivative-bound}
		\|R_\psi(V'f)\|_{W^{2,p}}\leq C\|E_h^{-1}V'f\|_{W^{1,p}}\leq Ch^{-1}\|f\|_{W^{1,p}}.
	\end{equation}
	The right inverses are evaluated on compactly supported inputs, so the higher local Sobolev estimate applies. Also, $S_\psi V:L^p\to W^{1,p}$ is uniformly bounded.
	
	Fix integers $K=N+1$ and $L=N+2$. The forcing $g_h=S_\psi(Va)$ is uniformly bounded in $W^{1,p}$. Since $T_h=O(h)$ on this space,
	\[
	\Big\|\sum_{k=K+1}^\infty T_h^kg_h\Big\|_{W^{1,p}}\leq C(Ch)^{K+1}\|g_h\|_{W^{1,p}}=O(h^{K+1}).
	\]
	Using the signs in \eqref{eq: def of rh} and \eqref{eq: def of sh}, we obtain
	\begin{equation}\label{eq: finite-Neumann-expansion}
		r_h=R_\psi V'\sum_{k=0}^KT_h^kg_h+O_{W^{2,p}}(h^K).
	\end{equation}
	Thus the infinite tail contributes $O(h^{N+1})$ in the required norm.
	
	We now expand each of the finitely many terms. Use the coefficient operators $B^j$ from \eqref{equ: d-bar inverse of f}, and define compactly supported smooth one-forms
	\[
	\omega_{0,j}=B^j(Va),\quad
	\omega_{k,j_1,\ldots,j_{k+1}}
	=B^{j_{k+1}}\big(VR(V'\omega_{k-1,j_1,\ldots,j_k})\big).
	\]
	The unweighted $R$ here means restriction of the fixed right inverse after extension by zero of its compactly supported input. Compact support of each new input follows from the multiplier $V$. The transport operators $B^j$ do not enlarge that support.
	
	For $k=0$, \eqref{equ: d-bar inverse of f} gives
	\[
	g_h=E_h\sum_{j=1}^Lh^j\omega_{0,j}+O_{W^{1,p}}(h^L).
	\]
	The weighted remainder is $h^LS_\psi$ applied to a fixed smooth compactly supported input, so it is $O(h^L)$ in $W^{1,p}$. The smoothing error has arbitrarily high order.
	
	Suppose the expansion has been proved for $k-1$. The cancellation of oscillatory factors inside the inner inverse is exact:
	\[
	R_\psi(V'E_h\omega)=R(V'\omega).
	\]
	Moreover, 
	\[
	T_h(E_h\omega)=S_\psi\big(VR(V'\omega)\big).
	\]
	Apply the second expansion in \eqref{equ: d-bar inverse of f} to this fixed smooth compactly supported input. The previous remainder is mapped to another $O(h^L)$ remainder in $W^{1,p}$, because $T_h$ is uniformly bounded there. Induction yields
	\begin{equation}\label{eq: iterated-Neumann-coefficients}
		T_h^kg_h
		=E_h\sum_{j_1,\ldots,j_{k+1}=1}^Lh^{j_1+\cdots+j_{k+1}}
		\omega_{k,j_1,\ldots,j_{k+1}}+O_{W^{1,p}}(h^L),
	\end{equation}
	for every fixed $0\leq k\leq K$. The constants may depend on $K,L,p$, which are fixed before $h\to0$.
	
	Apply $R_\psi V'$ to \eqref{eq: iterated-Neumann-coefficients}. The oscillatory factor again cancels in every polynomial term, while \eqref{eq: weighted-two-derivative-bound} bounds the error by $O_{W^{2,p}}(h^{L-1})$. Thus,
	\begin{equation}\label{eq: final-Neumann-coefficients}
		r_h=\sum_{k=0}^K\sum_{j_1,\ldots,j_{k+1}=1}^L
		h^{j_1+\cdots+j_{k+1}}R(V'\omega_{k,j_1,\ldots,j_{k+1}})
		+O_{W^{2,p}}(h^{N+1}).
	\end{equation}
	We used $K=N+1$ and $L=N+2$ for both errors.
	
	Every summand has degree at least $k+1$. Therefore, for $1\leq j\leq N$, its coefficient is the finite sum
	\[
	a_j=\sum_{k=0}^{j-1}\ \sum_{\substack{j_1,\ldots,j_{k+1}\geq1\\j_1+\cdots+j_{k+1}=j}}
	R(V'\omega_{k,j_1,\ldots,j_{k+1}}).
	\]
	These coefficients are smooth and do not depend on $N$. The remaining terms of the finite polynomial have degree at least $N+1$ and are absorbed in the $W^{2,p}$ error. This proves the first expansion.
	
	For the antiholomorphic expansion, apply the preceding proof to the operator with potential $\overline q$ and phase $-\Phi$, and take complex conjugates as in \eqref{eq:formula_for_tilde_rh}. Conjugation preserves the Sobolev norms and finite polynomial degrees. The right inverses, coefficients and signs become exactly those in \eqref{eq: form of tilderh}. This proves the second expansion for the same antiholomorphic fixed-point family.
\end{proof}

For the planar integral calculation, choose the standard scalar right inverses
\begin{equation}\label{eq: scalar-Cauchy}
	\mathscr C f(z)=\frac1\pi\int_\C\frac{f(w)}{z-w}\,dA(w),\quad
	\mathscr T f(z)=\frac1\pi\int_\C\frac{f(w)}{\bar z-\bar w}\,dA(w).
\end{equation}
They satisfy $\overline\partial\mathscr C f=f$ and $\partial\mathscr T f=f$. In the Euclidean metric, $R(f\,d\bar z)=\mathscr C f$ and $Sf=-\frac12\mathscr T f\,d\bar z$. This factor $-1/2$ follows from $\overline\partial^*(f\,d\bar z)=-2\partial f$.

For the planar construction, extend $X,q$ smoothly with compact support to $\C$, without changing them near the closed domain where the equation is to be solved. Write $X=(X^1,X^2)$, $b_\pm=X^1\pm\mathsf{i}X^2$, $a_+=F_A^{-1}$ and $a_-=F_{\overline A}$. Then
\[
\overline\partial\log a_+=\frac{b_+}{4},\quad
\partial\log a_-=\frac{b_-}{4}.
\]
With the standard inverses, these amplitudes may be chosen as $a_+=\exp\mathscr C(\frac{b_+}{4})$ and $a_-=\exp\mathscr T(\frac{b_-}{4})$. For a potential $q$, set $q_+=q-\partial b_++\frac{b_+b_-}{4}$ and $q_-=q-\overline\partial b_-+\frac{b_+b_-}{4}$.

Writing the one-form correction as its scalar coefficient, the holomorphic equations become
\begin{equation}\label{eq: scalar-holomorphic-system}
	r_h=\mathscr C(E_h^{-1}V_1s_h),\quad
	s_h=\mathscr T(E_hV_0(a+r_h)),\quad
	V_1=\frac{a_-}{a_+},\quad V_0=\frac{q_+a_+}{4a_-}.
\end{equation}
For the antiholomorphic construction they become
\begin{equation}\label{eq: scalar-antiholomorphic-system}
	\widetilde r_h=\mathscr T(E_h^{-1}W_1\widetilde s_h),\quad
	\widetilde s_h=\mathscr C(E_hW_0(a+\widetilde r_h)),\quad
	W_1=\frac{a_+}{a_-},\quad W_0=\frac{q_-a_-}{4a_+}.
\end{equation}
As before, the four coefficients have fixed smooth compactly supported cutoffs on a larger planar domain $D$, equal to one near the domain where the equation is solved. In particular, the displayed ratios hold near that domain. These formulas are the scalar specializations of the surface construction, with its signs and adjoint normalization included.

The scalar kernels are antisymmetric under interchange of their two variables. Hence, for compactly supported inputs,
\begin{equation}\label{eq:integration by parts for d}
	\int_\C(\mathscr C f)k\,dA=-\int_\C f\mathscr Ck\,dA,\quad
	\int_\C(\mathscr T f)k\,dA=-\int_\C f\mathscr Tk\,dA.
\end{equation}
For smooth inputs, Fubini's theorem applies because the kernel is locally integrable. The Sobolev versions follow by approximation.

\begin{lemma}\label{lemma: critical-pairing}
	Let $\widetilde r_h,\widetilde s_h$ be the antiholomorphic corrections in \eqref{eq: scalar-antiholomorphic-system}, with amplitude one and Morse phase. For $f\in W^{1,p}(\C)$, $p>2$, supported in a fixed compact subset of $D$, there is $\epsilon>0$ such that
	\begin{equation}\label{eq: critical-pairing-estimate}
		\bigg|\int_D E_hf\widetilde r_h\,dA\bigg|\leq Ch^{1+2\epsilon}\|f\|_{W^{1,p}}.
	\end{equation}
\end{lemma}

\begin{proof}
	The exact equation for $\widetilde r_h$ and \eqref{eq:integration by parts for d} give
	\[
	\int_D E_hf\widetilde r_h\,dA	=-\int_D\mathscr T(E_hf)E_h^{-1}W_1\widetilde s_h\,dA.
	\]
	The first factor is $O_{L^2}(h^{1/2+\epsilon}\|f\|_{W^{1,p}})$ by \eqref{eq:sobo_decayL2} with the appropriate sign. The second is $O_{L^2}(h^{1/2+\epsilon})$ by \eqref{eq:rh_L2_estim}. Choose the smaller exponent and apply Cauchy--Schwarz. The estimate uses the entire correction $\widetilde s_h$.
\end{proof}

For $L=-\Delta+X\cdot\nabla$, choose the extension of the adjoint drift to be $-X$ and its potential to be $-\nabla\cdot X$. We may take $\alpha_*=-\alpha$, so the adjoint holomorphic amplitude is exactly $F_{A^*}^{-1}=F_A=a_+^{-1}$.

\begin{lemma}[CGO solutions]\label{lemma: CGO solution summarize}
	For small $h>0$, the following solutions exist on a neighborhood of the closed domain under consideration.
	\begin{enumerate}[\rm(i)]
		\item The equations \eqref{linear general} and \eqref{linear adjoint general} have solutions $F_{A_1}^{-1}e^{\Phi_1/h}(1+r_1)$ and $F_{A^*}^{-1}e^{\Phi^*/h}(1+r_*)$. If the two phases have no critical points on the larger construction domain, their remainders satisfy \eqref{eq:r_h_expansion}. Their leading amplitudes can be chosen as $a_+$ and $a_+^{-1}$.
		\item For a holomorphic Morse phase $\Phi_2$, \eqref{linear general} has a solution $F_{\overline{A_2}}e^{-\overline{\Phi_2}/h}(1+\widetilde r_2)$. Its exact equations are \eqref{eq: scalar-antiholomorphic-system}, with the imaginary part of this $\Phi_2$, and its remainder satisfies \eqref{eq:rh_L2_estim}, \eqref{CZ estimate for r_2} and \eqref{eq: critical-pairing-estimate}.
	\end{enumerate}
\end{lemma}

\begin{proof}
	Apply the holomorphic and antiholomorphic constructions with amplitude one. Lemma~\ref{lem: asymptotic of remainders with linear} gives the noncritical expansions. The preceding choice of the adjoint extension gives the reciprocal leading amplitude. Each construction uses its own phase.
\end{proof}

\section{Carleman estimate and unique continuation}\label{sec: Carleman and UCP}

We prove a continuation result for the nonlocal equation obtained below. The Cauchy term is included as a component of a local system before applying the Carleman estimate.

\begin{lemma}[Unique continuation property]\label{lemma: UCP with Carleman}
	Let $U\subset\C$ be bounded and connected. Let $A_0,\alpha,\delta,\beta,\gamma$ be smooth in $U$, with $A_0$ nowhere zero, and let $\rho\in C_c^2(U)$ satisfy
	\begin{equation}\label{eq:ucp_equation}
		\overline\partial(A_0\overline\partial\rho+\alpha\rho)+\delta\rho
		=\beta\big(\mathscr C(\gamma\rho)+P\big)\quad\text{in }U,
	\end{equation}
	where $P$ is holomorphic and $\gamma\rho$ is extended by zero outside $U$. Suppose that $\rho=0$ on a nonempty open set $W\subset U$ and $\beta(z_0)\neq0$ for some $z_0\in W$. Then $\rho=0$ in $U$.
\end{lemma}

We use the two-parameter estimate
\begin{equation}\label{eq:2_parameter_carleman}
	\epsilon^{-1}\|e^{-\tau\phi_\epsilon}v\|_{L^2}^2
	\leq\|e^{-\tau\phi_\epsilon}\overline\partial v\|_{L^2}^2
\end{equation}
for compactly supported $v$ in an annulus, where
\begin{equation}\label{phi_epsilon (z)}
	\phi_\epsilon(z)=\varphi(z)-\frac{|z|^2}{2\epsilon\tau}
\end{equation}
and $\varphi$ is real and harmonic there. This is a scalar estimate and also holds componentwise for vectors.

We include a verification of the sign and constant. Put $\Psi=\tau\phi_\epsilon$ and $\mathcal P=e^{-\Psi}\overline\partial e^\Psi=\overline\partial+\overline\partial\Psi$. Its Hermitian adjoint is $\mathcal P^*=-\partial+\partial\Psi$. Integration by parts for compactly supported functions gives
\[
\|\mathcal P f\|_2^2-\|\mathcal P^*f\|_2^2
=(f,[\mathcal P^*,\mathcal P]f)_{L^2}
=-\frac12\int\Delta\Psi\,|f|^2\,dA=\epsilon^{-1}\|f\|_2^2.
\]
Here $\Delta\Psi=-2/\epsilon$. Taking $f=e^{-\Psi}v$ proves \eqref{eq:2_parameter_carleman}. Approximation extends the estimate to compactly supported $W^{1,2}$ functions in the annulus.

\begin{proof}[Proof of Lemma \ref{lemma: UCP with Carleman}]
	Set $p=A_0\overline\partial\rho+\alpha\rho$ and $t=\mathscr C(\gamma\rho)+P$. On a sufficiently small disk about $z_0$, $\beta$ is nonzero. Since $\rho$ and its derivatives vanish there, $p=0$, and \eqref{eq:ucp_equation} gives $t=0$ there as well. Thus, all three quantities vanish on a disk. The Cauchy integral is not assumed to have the support of $\rho$.
	
	We show that this vanishing propagates locally. Work in a disk $B_R\Subset U$ centered at a point of a disk where $\rho,p,t$ vanish, and translate its center to zero. On a neighborhood of $\overline{B_R}$ choose $\theta$ satisfying $\overline\partial\theta=A_0^{-1}(\alpha-\overline\partial A_0)$. Such a function is obtained by applying the Cauchy right inverse to a compactly supported extension of the right-hand side. Set $u=e^\theta A_0\rho$. Direct differentiation gives
	\[
	\overline\partial u=e^\theta\big(A_0\overline\partial\rho+(\overline\partial A_0+A_0\overline\partial\theta)\rho\big)=e^\theta p.
	\]
	Writing $a=A_0^{-1}e^{-\theta}$, the equation becomes
	\begin{equation}\label{eq:ucp_equation_conjugated}
		\overline\partial(e^{-\theta}\overline\partial u)+\delta a u=\beta t,\quad
		\overline\partial t=\gamma a u.
	\end{equation}
	Thus, $Y=(u,p,t)^T$ satisfies
	\begin{equation}\label{eq:ucp_equation_conjugated_redefined}
		\overline\partial Y=MY,\quad
		M=\begin{pmatrix}0&e^\theta&0\\-\delta a&0&\beta\\\gamma a&0&0\end{pmatrix}.
	\end{equation}
	The matrix is bounded on $\overline{B_R}$. The components belong locally to $W^{1,2}$, so the Carleman estimate applies after multiplication by a cutoff.
	
	Let $0<R_0<R_1<R_2<R$ and choose $\chi\in C_c^\infty(B_{R_2})$ equal to one on $B_{R_1}$. Since $Y$ vanishes in a neighborhood of the origin, $\chi Y$ is compactly supported away from zero. Choose
	\begin{equation}\label{varphi-harmonic}
		\varphi(z)=\log|z|^2.
	\end{equation}
	Apply \eqref{eq:2_parameter_carleman} to the three components:
	\begin{equation}\label{eq:estimate1_for_ucp}
		\epsilon^{-1}\|e^{-\tau\phi_\epsilon}\chi Y\|_2^2
		\leq\|e^{-\tau\phi_\epsilon}\overline\partial(\chi Y)\|_2^2.
	\end{equation}
	The system and the product rule give
	\begin{equation}\label{eq:estimate2_for_ucp}
		\overline\partial(\chi Y)=\chi MY+(\overline\partial\chi)Y.
	\end{equation}
	Consequently,
	\begin{equation}\label{UCP derive 1}
		\epsilon^{-1}\|e^{-\tau\phi_\epsilon}\chi Y\|_2^2
		\leq2\|M\|_\infty^2\|e^{-\tau\phi_\epsilon}\chi Y\|_2^2
		+2\|e^{-\tau\phi_\epsilon}(\overline\partial\chi)Y\|_2^2.
	\end{equation}
	Choose $\epsilon>0$ so that $2\epsilon\|M\|_\infty^2\leq1/2$, and fix it. Absorption yields
	\begin{equation}\label{eq:final_ineq_absorbed}
		\|e^{-\tau\phi_\epsilon}\chi Y\|_2^2
		\leq4\epsilon\|e^{-\tau\phi_\epsilon}(\overline\partial\chi)Y\|_2^2.
	\end{equation}
	
	With the sign in \eqref{phi_epsilon (z)}, the weight is $e^{-\tau\phi_\epsilon}=|z|^{-2\tau}e^{|z|^2/(2\epsilon)}$. Its square therefore contains $|z|^{-4\tau}$. On the support of $\overline\partial\chi$, $R_1\leq|z|\leq R_2$, while $\chi=1$ on $B_{R_0}$. Hence
	\begin{equation}\label{eq:almost_final_estim}
		R_0^{-4\tau}\|Y\|_{L^2(B_{R_0})}^2
		\leq4\epsilon\|\overline\partial\chi\|_\infty^2
		e^{R_2^2/\epsilon}R_1^{-4\tau}
		\|Y\|_{L^2(B_{R_2}\setminus B_{R_1})}^2.
	\end{equation}
	The right-hand norm is finite and independent of $\tau$. Since $R_0/R_1<1$, multiplication by $R_0^{4\tau}$ and passage to $\tau\to\infty$ give $Y=0$ on $B_{R_0}$. As $R_0<R$ was arbitrary after choosing intermediate radii, $Y=0$ on $B_R$. This is equivalent to $\rho=p=t=0$ there.
	
	For completeness, let $Z$ be the set of points having a neighborhood where $\rho,p,t$ vanish. It is nonempty and open. If $x$ lies in its closure in $U$, choose a disk about $x$ compactly contained in $U$ and a point $y\in Z$ sufficiently close to $x$. A disk centered at $y$, still compactly contained in $U$, contains $x$ in its interior. The local argument just proved enlarges the zero set to that disk, so $x\in Z$. Thus $Z$ is also closed in $U$, and connectedness gives $Z=U$. The lemma follows.
\end{proof}

\section{Analysis of the second integral identity}\label{sec: the second integral identity}
Section~\ref{sec: unique of diffeo} gives $g_1=cg_2$ in the original coordinates. We use the second linearized equation to determine $c$. Two features must be included in the calculation: the volume form changes under conformal scaling, and a Euclidean Hessian acquires first derivative terms under a nonlinear coordinate change.

Let $u_0^{(j)}$ solve \eqref{MA equation zero boundary j=1,2}, and set $g_j=D^2u_0^{(j)}$. The second linearized equation is
\begin{equation}\label{2nd Linearized MA equation DN}
	\begin{cases}
		(-\Delta_{g_j}+X_{g_j}\cdot\nabla)w_j
		=-\tr(g_j^{-1}D^2v_j^{(1)}g_j^{-1}D^2v_j^{(2)})&\text{in }\Omega,\\
		w_j=0&\text{on }\partial\Omega.
	\end{cases}
\end{equation}

\begin{lemma}\label{lemma: second linearization det}
	The nonlinear DN map and the prescribed boundary source jets determine $\partial_{\nu_g}w|_{\partial\Omega}$. Under \eqref{eq: same DN map}, the two conormal derivatives agree for common boundary directions.
\end{lemma}

\begin{proof}
	By \eqref{eq: measured-first-second-variations}, $\partial_\nu w=D^2\Lambda_F(0)[\phi_1,\phi_2]$ on $\p \Omega$. Since $w=0$ also on $\p\Omega$, its tangential derivative vanishes and $\nabla w=\nu\partial_\nu w$. Thus,
	\begin{equation}\label{eq: second-DN-conormal}
		\partial_{\nu_g}w
		=\sqrt{|g|}(\nu^Tg^{-1}\nu)D^2\Lambda_F(0)[\phi_1,\phi_2].
	\end{equation}
	The boundary metric is known by Lemma~\ref{lemma: boundary determination}. For two sources, it and the second derivative of the nonlinear DN map agree. This proves the lemma.
\end{proof}

\subsection{The second integral identity}
Write
\begin{equation*}
	\mathcal L_g=-\Delta_g+X_g\cdot\nabla \quad \text{and} \quad Q_g(v^{(1)},v^{(2)})=\tr(g^{-1}D^2v^{(1)}g^{-1}D^2v^{(2)}).
\end{equation*} 
The transpose with respect to the bilinear pairing in $dV_g$ is $\mathcal L_g^\dagger=-\Delta_g-\operatorname{div}_g(X_g\,\cdot\,)$. We consider
\begin{equation}\label{adjoint of 1st linearization}
	\begin{cases}
		\Delta_gv^*+\dfrac1{\sqrt{|g|}}\partial_b(\sqrt{|g|}X_g^bv^*)=0&\text{in }\Omega,\\
		v^*=\varphi^*&\text{on }\partial\Omega.
	\end{cases}
\end{equation}
This Dirichlet problem is solvable for smooth $\varphi^*$. Indeed, the Dirichlet realization of $\mathcal L_g$ is Fredholm of index zero and has trivial kernel by the maximum principle. Its transpose Dirichlet realization is therefore invertible. Extending $\varphi^*$ smoothly to the interior reduces \eqref{adjoint of 1st linearization} to a zero boundary problem; elliptic regularity gives a smooth solution. The argument applies to complex data by complexification.

Let $w$ solve $\mathcal L_gw=-Q_g$ with zero boundary value. Using the conormal density defined in Section~\ref{sec: unique of diffeo}, integration by parts gives
\begin{equation}\label{integral id 1}
	\begin{aligned}
		\int_{\partial\Omega}\varphi^*\partial_{\nu_g}w\,dS
		&=\int_\Omega v^*\Delta_gw\,dV_g+\int_\Omega g^{ab}\partial_av^*\partial_bw\,dV_g=\int_\Omega v^*\Delta_gw\,dV_g-\int_\Omega w\Delta_gv^*\,dV_g\\
		&=\int_\Omega v^*(X_g\cdot\nabla w+Q_g)\,dV_g+\int_\Omega w\operatorname{div}_g(X_gv^*)\,dV_g=\int_\Omega v^*Q_g\,dV_g.
	\end{aligned}
\end{equation}
The integration by parts involving $\nabla v^*$ gives the boundary term $\int_{\partial\Omega}w\partial_{\nu_g}v^*\,dS=0$. The two drift terms combine into the integral of $\operatorname{div}_g(wX_gv^*)$, whose boundary term also vanishes because $w=0$. The factor $\sqrt{|g|}$ is already included in $\partial_{\nu_g}w$.

\begin{lemma}[Integral identity for the second linearization]~
	\begin{enumerate}[\rm(i)]
		\item\label{item 1 integral id} For the solutions above,
		\begin{equation}\label{integral id second 1}
			\int_{\partial\Omega}\varphi^*\partial_{\nu_g}w\,dS
			=\int_\Omega v^*\tr(g^{-1}D^2v^{(1)}g^{-1}D^2v^{(2)})\,dV_g.
		\end{equation}
		\item\label{item 2 integral id} Under \eqref{eq: same DN map}, for common first-variation boundary directions and adjoint solutions with common boundary trace,
		\begin{equation}\label{integral id second 2}
			\int_\Omega v_1^*Q_{g_1}(v_1^{(1)},v_1^{(2)})\,dV_{g_1}
			-\int_\Omega v_2^*Q_{g_2}(v_2^{(1)},v_2^{(2)})\,dV_{g_2}=0.
		\end{equation}
	\end{enumerate}
\end{lemma}

\begin{proof}
	The first assertion is \eqref{integral id 1}. For the second, apply the first identity to each source. The adjoint boundary traces agree, and Lemma~\ref{lemma: second linearization det} gives the same conormal derivative of $w_j$. Subtract the equal boundary integrals.
\end{proof}

Since $g_1=cg_2$, we have $g_2^{-1}=cg_1^{-1}$ and $dV_{g_2}=c^{-1}dV_{g_1}$ in dimension two. Hence, $\mathcal L_{g_2}=c\mathcal L_{g_1}$, and the first linearized solutions with common boundary values agree. Their adjoint solution spaces also agree. For compactly supported $f,k$,
\[
\int k\mathcal L_{g_2}f\,dV_{g_2}=\int k\mathcal L_{g_1}f\,dV_{g_1}=\int f\mathcal L_{g_1}^\dagger k\,dV_{g_1}=\int f(c\mathcal L_{g_1}^\dagger k)\,dV_{g_2}.
\]
Thus, $\mathcal L_{g_2}^\dagger=c\mathcal L_{g_1}^\dagger$. We may use common functions $v^{(1)},v^{(2)},v^*$ in the two integrals.

\subsection{Change of variables for the Hessian}

The Hessians in \eqref{integral id second 2} are Euclidean Hessians, not covariant Hessians of $g_j$. We compute their change of variables before evaluating the integral.

Let $\chi$ be any smooth coordinate diffeomorphism and put $\widetilde u=u\circ\chi$. Componentwise differentiation gives
\begin{equation}\label{eq: Hessian-chain-rule}
	\begin{aligned}
		\partial_{ab}\widetilde u=\partial_a\big((\partial_iu\circ\chi)\partial_b\chi^i\big)=(\partial_{ji}u\circ\chi)\partial_a\chi^j\partial_b\chi^i
		+(\partial_iu\circ\chi)\partial_{ab}\chi^i.
	\end{aligned}
\end{equation}
Write $J_\chi=D\chi$. Since $\nabla u\circ\chi=J_\chi^{-T}\nabla\widetilde u$, define
\begin{equation}\label{eq: coordinate-Hessian}
	C^k_{ab}=-(J_\chi^{-1})^k_i\partial_{ab}\chi^i,\quad
	H_C\widetilde u=D^2\widetilde u+C^k\partial_k\widetilde u.
\end{equation}
Then $H_C\widetilde u=J_\chi^T(D^2u\circ\chi)J_\chi$. The symmetry of the second derivatives of $\chi$ makes every $C^k$ symmetric. For $\widetilde g=\chi^*g=J_\chi^T(g\circ\chi)J_\chi$, the inverse is $\widetilde g^{-1}=J_\chi^{-1}(g^{-1}\circ\chi)J_\chi^{-T}$. Hence
\[
(g^{-1}D^2u)\circ\chi=J_\chi\widetilde g^{-1}H_C\widetilde u\,J_\chi^{-1}.
\]
Multiplication of the two factors and cyclicity of the trace give the full formula
\begin{equation}\label{change of variable in the trace}
	\begin{aligned}
		\tr(g^{-1}D^2v^{(1)}g^{-1}D^2v^{(2)})\circ\chi
		={}&\tr(\widetilde g^{-1}D^2\widetilde v^{(1)}\widetilde g^{-1}D^2\widetilde v^{(2)})\\
		&+\tr(\widetilde g^{-1}D^2\widetilde v^{(1)}\widetilde g^{-1}C^k\partial_k\widetilde v^{(2)})\\
		&+\tr(\widetilde g^{-1}C^k\partial_k\widetilde v^{(1)}\widetilde g^{-1}D^2\widetilde v^{(2)})\\
		&+\tr(\widetilde g^{-1}C^k\partial_k\widetilde v^{(1)}\widetilde g^{-1}C^\ell\partial_\ell\widetilde v^{(2)}).
	\end{aligned}
\end{equation}
The matrices $C^k$ contain the inverse Jacobian factors and the second derivatives of $\chi$.

Now let $\chi:\widetilde\Omega\to\Omega$ be the isothermal map for $g_1$, so $\chi^*g_1=\mu I_2$. Put $\mathbf v^{(k)}=v^{(k)}\circ\chi$ and $\mathbf v^*=v^*\circ\chi$. Define
\[
P=\tr(D^2\mathbf v^{(1)}D^2\mathbf v^{(2)}),\quad
M_1=\tr(D^2\mathbf v^{(1)}C^k\partial_k\mathbf v^{(2)}),
\]
and $M_2=\tr(C^k\partial_k\mathbf v^{(1)}D^2\mathbf v^{(2)})$, $M_3=\tr(C^k\partial_k\mathbf v^{(1)}C^\ell\partial_\ell\mathbf v^{(2)})$. Formula \eqref{change of variable in the trace} gives
\begin{equation}\label{eq: first-metric-coordinate-integrand}
	Q_{g_1}(v^{(1)},v^{(2)})\circ\chi=\mu^{-2}(P+M_1+M_2+M_3).
\end{equation}

Theorem~\ref{theorem: J=Id} gives $g_1=cg_2$ in the original coordinates. We therefore use the same map $\chi$ for both metrics, so the matrices $C^k$, which depend only on $\chi$, are the same in both transformed Hessians. Write $c$ for $c\circ\chi$ in these coordinates. Since $\chi^*g_2=c^{-1}\mu I_2$,
\begin{equation}\label{eq: second-metric-coordinate-integrand}
	Q_{g_2}(v^{(1)},v^{(2)})\circ\chi=c^2\mu^{-2}(P+M_1+M_2+M_3).
\end{equation}
The volume forms transform separately:
\[
\chi^*dV_{g_1}=\mu\,dA,\quad \chi^*dV_{g_2}=c^{-1}\mu\,dA.
\]
These include the Jacobian of $\chi$, through $\sqrt{\det(\chi^*g_j)}=\sqrt{\det(g_j\circ\chi)}|\det D\chi|$.

Substituting both transformed integrands and both volume forms into \eqref{integral id second 2} gives
\[
0=\int_{\widetilde\Omega}\mu^{-1}\mathbf v^*(P+M_1+M_2+M_3)\,dA
-\int_{\widetilde\Omega}c\mu^{-1}\mathbf v^*(P+M_1+M_2+M_3)\,dA.
\]
Therefore define
\begin{equation}\label{def of G}
	G=\mu^{-1}(1-c).
\end{equation}
The correction is
\begin{equation}\label{Y=lower order terms}
	Y=G\mathbf v^*\big[
	\tr(D^2\mathbf v^{(1)}C^k\partial_k\mathbf v^{(2)})
	+\tr(C^k\partial_k\mathbf v^{(1)}D^2\mathbf v^{(2)})
	+\tr(C^k\partial_k\mathbf v^{(1)}C^\ell\partial_\ell\mathbf v^{(2)})
	\big].
\end{equation}
Equivalently, with
\begin{equation}\label{def of mathcal C}
	\mathcal C^k=(1-c)C^k,
\end{equation}
the first two terms of $\mu Y$ use $\mathcal C^k$, and the last uses $(1-c)C^kC^\ell$. All three types of coordinate contributions are present.

For any smooth function $v$, contraction of \eqref{eq: Hessian-chain-rule} gives
\[
(g_1^{ab}\partial_{ab}v)\circ\chi=\mu^{-1}\tr(H_C\mathbf v).
\]
On the other hand, the operator transformation giving \eqref{equ: isothermal 2} identifies the same expression with $-\mu^{-1}L\mathbf v$. Thus $\tr(H_C\mathbf v)=-L\mathbf v$ for every smooth $\mathbf v$, which gives $X^k=-\tr(C^k)$. We use this identity below.

\begin{theorem}\label{theorem: uniqueness of conformal factor}
	Let $\widetilde\Omega\subset\R^2$ be a bounded domain with smooth boundary. Let $\mu>0$, $c>0$ and the real symmetric matrices $C^1,C^2$ be smooth on $\overline{\widetilde\Omega}$. Set $X^k=-\tr(C^k)$, $L=-\Delta+X\cdot\nabla$, and $L^t=-\Delta-X\cdot\nabla-\nabla\cdot X$. Assume that $c-1$ and its derivatives through order two vanish on the boundary. Let $G,Y$ be given by \eqref{def of G} and \eqref{Y=lower order terms}. If
	\begin{equation}\label{integral id.}
		\int_{\widetilde\Omega}\big[G\mathbf v^*\tr(D^2\mathbf v^{(1)}D^2\mathbf v^{(2)})+Y\big]\,dA=0
	\end{equation}
	for all smooth solutions of $L\mathbf v^{(1)}=L\mathbf v^{(2)}=0$ and $L^t\mathbf v^*=0$, then $c=1$ in $\widetilde\Omega$.
\end{theorem}

\begin{remark}\label{remark: large U}
	The boundary determination gives equality of the jets of $g_1,g_2$ through order two. In $g_1=cg_2$, the boundary values first give $c=1$. One differentiation then gives $(\partial_kc)g_2=0$, hence $\partial_kc=0$. Two differentiations give $(\partial_{k\ell}c)g_2=0$ after the already known terms cancel. Thus, all boundary derivatives of $c-1$ through order two vanish. The same is true of $G$.
	
	Extend $G$ by zero to a larger connected domain $U\Supset\overline{\widetilde\Omega}$. This extension is $C_c^2(U)\cap W^{3,\infty}(U)$. Distributional derivatives through order three produce no boundary measures, because the jumps of the function and of its derivatives through order two are zero. The remaining piecewise third derivatives are bounded.
	
	Extend $C^1,C^2$ smoothly outside $\widetilde\Omega$ and define $X^k=-\tr(C^k)$ there. Choose a disk $B_0\Subset U\setminus\overline{\widetilde\Omega}$ and a real $\kappa\neq0$. By a cutoff supported outside $\overline{\widetilde\Omega}$, arrange $C^1=-(\kappa/2)I_2$ and $C^2=0$ near $\overline{B_0}$. Then $X=(\kappa,0)$ there. Extend these coefficients farther to the domains needed for the CGO construction, without changing them near $\overline U$.
	
	The constructed solutions restrict to smooth solutions of the required linearized equations on $\widetilde\Omega$. Thus,
	\begin{equation}\label{integral id._extended}
		\int_U\big[G\mathbf v^*\tr(D^2\mathbf v^{(1)}D^2\mathbf v^{(2)})+Y\big]\,dA=0.
	\end{equation}
	The values of the extension outside $\widetilde\Omega$ do not change this identity because $G=Y=0$ there. Outside the original domain, the tensors need only satisfy $X^k=-\tr(C^k)$; they need not come from an extended coordinate map.
\end{remark}

\section{Asymptotic analysis for the second integral identity}\label{sec: asymptotic analysis}

We compute the seven groups in the second integral identity separately. The contributions in which derivatives fall on the amplitudes are included before taking limits.

For $f\in C_c^\infty(\R^2)$, the stationary phase formula for the positive phase is
\begin{equation}\label{stationary phase expansion}
	\frac1{2\pi h}\int_{\R^2}e^{\mathsf{i}x_1x_2/h}f(x)\,dx
	=\sum_{k=0}^{N-1}\frac{h^k}{k!}(\overline\partial^2-\partial^2)^kf(0)+R_N(f;h).
\end{equation}
The sign follows from $\overline\partial^2-\partial^2=\mathsf{i}\partial_{x_1}\partial_{x_2}$. For the Fourier convention $\widehat f(\zeta)=\int e^{-\mathsf{i}x\cdot\zeta}f(x)\,dx$, the corresponding multiplier is $e^{-\mathsf{i}h\zeta_1\zeta_2}$. Taylor's formula gives
\[
|R_N(f;h)|\leq\frac{h^N}{(2\pi)^2N!}\int_{\R^2}|\zeta_1\zeta_2|^N|\widehat f(\zeta)|\,d\zeta.
\]
The Fourier identity can first be regularized by Gaussian cutoffs and then passed to the limit.

\subsection{Reading guide for the asymptotic analysis}

We use the decomposition
\begin{equation}\label{eq:integral_to_analyze}
	I_h=\int_U\big[G\mathbf v^*\tr(D^2\mathbf v^{(1)}D^2\mathbf v^{(2)})+Y\big]\,dA=\sum_{j=1}^7S_j.
\end{equation}
Section~\ref{sec:Step1} specifies the phases, the amplitudes and the finite expansions. Section~\ref{sec:Step21} computes $S_1$, including $S_{1,1}$ and $S_{1,2}$. Section~\ref{sec:Step22} computes the local contributions of $S_5,S_6$, and Section~\ref{sec:Step23} computes their critical correction terms. Section~\ref{sec:Step3} proves the vanishing of the remaining groups. In particular, the nonlocal limits of $S_{5,3}$ and $S_{6,3}$ must be added; $S_{6,3}$ is not discarded.

\subsection{Step 1. Initialization and decomposition of the integral identity}\label{sec:Step1}

Fix $\xi\in U\setminus\partial\widetilde\Omega$. Let $D$ be a larger disk containing the supports of the coefficient cutoffs in the planar CGO construction, and choose $\lambda>\operatorname{diam}(D)$. With $w=z-\xi$, put
\begin{equation}\label{Phi_1 and Phi_2}
	\Phi_1=\lambda w+\frac{w^2}{8},\quad
	\Phi_2=-\frac{w^2}{4},\quad
	\Phi^*=-\lambda w+\frac{w^2}{8}.
\end{equation}
The first and third phases are noncritical on $\overline D$. Put $s=-\Phi_2=w^2/4$ and
\begin{equation}\label{weight function psi}
	E_h=e^{(s-\bar s)/h}
	=e^{\mathsf{i}(x_1-\operatorname{Re}\xi)(x_2-\operatorname{Im}\xi)/h}.
\end{equation}
The antiholomorphic construction uses the holomorphic phase $s$, so its exponential is $e^{-\bar s/h}=e^{\overline{\Phi_2}/h}$.

Using the amplitudes specified in Section~\ref{sec: CGOs}, take
\[
\mathbf v^{(1)}=a_+e^{\Phi_1/h}(1+r_1),\quad
\mathbf v^{(2)}=a_-e^{\overline{\Phi_2}/h}(1+\widetilde r_2),
\]
and
\begin{equation}\label{v_ast}
	\mathbf v^*=a_+^{-1}e^{\Phi^*/h}(1+r_*).
\end{equation}
The three exponential factors multiply to $E_h$. The critical correction satisfies the exact scalar equations
\begin{equation}\label{eq: representation of r_1 and r_2}
	\widetilde r_2=\mathscr T(E_h^{-1}W_1\widetilde s_2),\quad
	\widetilde s_2=\mathscr C(E_hW_0(1+\widetilde r_2)),\quad
	\partial\widetilde r_2=E_h^{-1}W_1\widetilde s_2.
\end{equation}
On $U$, $W_1=a_+/a_-$ and $W_0=q_-a_-/(4a_+)$, where $q_-=-\overline\partial b_-+b_+b_-/4$. The equations for $r_1,r_*$ use their own phases, not the weight \eqref{weight function psi}.

Write $v_{10}=a_+e^{\Phi_1/h}$, $v_{11}=a_+e^{\Phi_1/h}r_1$, $v_{20}=a_-e^{\overline{\Phi_2}/h}$ and $v_{21}=a_-e^{\overline{\Phi_2}/h}\widetilde r_2$. Bilinearity gives
\begin{equation}\label{eq:S1-S4}
	\tr(D^2\mathbf v^{(1)}D^2\mathbf v^{(2)})
	=\sum_{\ell,m=0}^1\tr(D^2v_{1\ell}D^2v_{2m}).
\end{equation}
Accordingly,
\begin{equation}\label{eq:S1toS4_terms}
	\int_U G\mathbf v^*\tr(D^2\mathbf v^{(1)}D^2\mathbf v^{(2)})\,dA=S_1+S_2+S_3+S_4,
\end{equation}
where
\begin{equation}\label{S_1 to S_4 definitions}
	\begin{aligned}
		S_1&=\int_U G\mathbf v^*\tr(D^2v_{10}D^2v_{20})\,dA,&
		S_2&=\int_U G\mathbf v^*\tr(D^2v_{11}D^2v_{20})\,dA,\\
		S_3&=\int_U G\mathbf v^*\tr(D^2v_{10}D^2v_{21})\,dA,&
		S_4&=\int_U G\mathbf v^*\tr(D^2v_{11}D^2v_{21})\,dA.
	\end{aligned}
\end{equation}
The other groups are
\begin{equation}\label{eq:S5toS7_terms}
	\begin{aligned}
		S_5&=\int_U G\mathbf v^*\tr(D^2\mathbf v^{(1)}C^k\partial_k\mathbf v^{(2)})\,dA,\\
		S_6&=\int_U G\mathbf v^*\tr(C^k\partial_k\mathbf v^{(1)}D^2\mathbf v^{(2)})\,dA,\\
		S_7&=\int_U G\mathbf v^*\tr(C^k\partial_k\mathbf v^{(1)}C^\ell\partial_\ell\mathbf v^{(2)})\,dA.
	\end{aligned}
\end{equation}
For $j=5,6$, let $S_{j,1},S_{j,2},S_{j,3},S_{j,4}$ denote the four terms obtained with $(v_{10},v_{20})$, $(v_{11},v_{20})$, $(v_{10},v_{21})$, $(v_{11},v_{21})$, respectively. Explicitly, for $\ell,m\in\{0,1\}$,
\[
S_{5,1+\ell+2m}=\int_U G\mathbf v^*\tr(D^2v_{1\ell}C^k\partial_kv_{2m})\,dA,\quad
S_{6,1+\ell+2m}=\int_U G\mathbf v^*\tr(C^k\partial_kv_{1\ell}D^2v_{2m})\,dA.
\]
Thus all four terms in each mixed group are specified.

We use the notation
\begin{equation}\label{useful elements in computations}
	p=\Phi_1'=\lambda+\frac w4,\quad
	q=(\Phi^*)'=-\lambda+\frac w4,\quad
	t=s'=\frac w2=p+q.
\end{equation}
In particular, $p'=1/4$, $t'=1/2$, $pq=-\lambda^2+w^2/16$, and $t(\xi)=0$. Define $C_\pm=C^1\pm\mathsf{i}C^2$ and
\[
\sigma=\frac14\tr(BC_+),\quad \tau=\frac14\tr(BC_-),\quad
e=\overline\tau=\frac14\tr(AC_+),\quad d=\overline\sigma=\frac14\tr(AC_-).
\]
Here $A,B$ are the matrices in \eqref{matrices A B}. Since $X^k=-\tr C^k$, one has $\tr C_+=-b_+$ and $\tr C_-=-b_-$.

Fix $N\geq4$ and a finite $p_0>2$. By Lemma~\ref{lem: asymptotic of remainders with linear}, each noncritical normalized amplitude is a finite polynomial plus an $O_{W^{2,p_0}}(h^{N+1})$ error. In all computations below, write $a_h,b_h$ for these finite amplitudes in the first and adjoint factors. For the full first solution, their constant terms are $a_+$ and $a_+^{-1}$. For the $r_1$ part alone, the first constant term is zero. We keep the critical correction exact.

We justify discarding only these high order tails. Conjugating an operator of order at most two by a fixed CGO exponential costs at most $Ch^{-2}$. The normalized tail in a Hessian is consequently $O_{L^2}(h^{N-1})$. The normalized Hessian of the critical solution satisfies
\[
\|e^{-\overline{\Phi_2}/h}D^2\mathbf v^{(2)}\|_2
\leq C\big(h^{-2}\|1+\widetilde r_2\|_2+h^{-1}\|\nabla\widetilde r_2\|_2+\|D^2\widetilde r_2\|_2\big)\leq Ch^{-2}.
\]
The estimates are local on a fixed set containing the support of $G$, using \eqref{CZ estimate for r_2}. The normalized adjoint tail is $O_{L^\infty}(h^{N+1})$ by Sobolev embedding. Factoring out the common product of exponentials, whose modulus is one, and applying Cauchy--Schwarz gives
\begin{equation}\label{eq: asymptotic-tail-bound}
	h|\text{error in any }S_j|
	\leq C h\big(h^{N-1}h^{-2}+h^{N+1}h^{-2}h^{-2}\big) \leq Ch^{N-2}\longrightarrow0 \quad \text{as }h\to 0.
\end{equation}
The mixed groups have fewer derivatives and obey the same upper bound. This argument does not estimate a growing CGO exponential by itself.

We will repeatedly use the following consequence of stationary phase. If $f\in W^{1,p_0}_c(D)$ is smooth near $\xi$, then
\begin{equation}\label{eq: leading-stationary-limit}
	h^{-1}\int_D E_hf\,dA\longrightarrow2\pi f(\xi) \quad \text{as }h\to 0.
\end{equation}
Choose a smooth cutoff equal to one near $\xi$ and apply \eqref{stationary phase expansion} to that part. On the complement, the real phase $\phi_\xi=(x_1-\operatorname{Re}\xi)(x_2-\operatorname{Im}\xi)$ has nonzero gradient. With $V=\nabla\phi_\xi/|\nabla\phi_\xi|^2$, one integration by parts writes the remaining integral divided by $h$ as an oscillatory integral with an $L^1$ amplitude $\mathsf{i}\operatorname{div}(Vf)$. Local coordinates with $\phi_\xi$ as one coordinate and the Riemann--Lebesgue lemma make it tend to zero. The same reasoning gives $\int E_hf\to0$ for fixed compactly supported $L^1$ amplitudes.

The coefficients below are finite polynomials in $h$, with fixed compact support and bounded $W^{1,p_0}$ norms. They are smooth near $\xi$ because $G$ is smooth away from $\partial\widetilde\Omega$. These facts permit application of \eqref{eq: leading-stationary-limit} coefficient by coefficient. The critical pairing estimate gives, uniformly for such families,
\begin{equation}\label{eq: pairing-limit-rule}
	h^{-1}\int_D E_hf_h\widetilde r_2\,dA=O(h^{2\epsilon})\longrightarrow0.
\end{equation}

\subsection{Step 2. Dominating contributions}\label{sec:Step2}

\subsubsection{\it Step 2-1. Dominating term I from $S_1$.}\label{sec:Step21}

The matrix identities used throughout are
\begin{equation}\label{useful matrices computations}
	A^2=B^2=0,\quad \tr A=\tr B=0,\quad \tr(AB)=4.
\end{equation}
They give the exact scalar formula
\begin{equation}\label{eq: full-Hessian-contraction}
	\tr(D^2uD^2v)=4\partial^2u\,\overline\partial^2v
	+4\overline\partial^2u\,\partial^2v
	+8\partial\overline\partial u\,\partial\overline\partial v.
\end{equation}
Thus all three products, including the mixed derivative product, are included.

Let $u_h=a_he^{\Phi_1/h}$ and $v_h^*=b_he^{\Phi^*/h}$ be finite-amplitude functions. In this calculation $v_h$ may be either $a_-e^{-\bar s/h}$ or $a_-e^{-\bar s/h}\widetilde r_2$. Transfer the two derivatives on $v_h$ in each term of \eqref{eq: full-Hessian-contraction}. Compact support of $G$ gives
\begin{equation}\label{eq: Hessian-transferred}
	\begin{aligned}
		\frac14\int Gv_h^*\tr(D^2u_hD^2v_h)\,dA
		=\int v_h\big[
		\overline\partial^2(Gv_h^*\partial^2u_h)
		+\partial^2(Gv_h^*\overline\partial^2u_h)
		+2\partial\overline\partial(Gv_h^*\partial\overline\partial u_h)
		\big]\,dA.
	\end{aligned}
\end{equation}
All interface terms vanish because the zero extension of $G$ is $C^2$.

Put $m_h=b_ha_h$. Product differentiation gives
\[
e^{-\Phi_1/h}\partial^2u_h=h^{-2}p^2a_h+h^{-1}(2p\partial a_h+p'a_h)+\partial^2a_h,
\]
while $e^{-\Phi_1/h}\overline\partial^2u_h=\overline\partial^2a_h$ and $e^{-\Phi_1/h}\partial\overline\partial u_h=h^{-1}p\overline\partial a_h+\partial\overline\partial a_h$. Also,
\[
e^{-s/h}\partial^2(e^{s/h}f)=\partial^2f+h^{-1}(2t\partial f+t'f)+h^{-2}t^2f,\quad
e^{-s/h}\partial\overline\partial(e^{s/h}f)=\partial\overline\partial f+h^{-1}t\overline\partial f.
\]
Substitution into \eqref{eq: Hessian-transferred} yields $4\int E_ha_-\zeta_h(h^{-2}H_2+h^{-1}H_1+H_0)\,dA$, where $\zeta_h=1$ or $\widetilde r_2$, and every coefficient is
\begin{align}
	H_2={}&\overline\partial^2(Gm_hp^2)+t^2Gb_h\overline\partial^2a_h+2t\overline\partial(Gb_hp\overline\partial a_h),\label{eq: Hessian-H2}\\
	H_1={}&\overline\partial^2\big(Gb_h(2p\partial a_h+p'a_h)\big)
	+2t\partial(Gb_h\overline\partial^2a_h)+t'Gb_h\overline\partial^2a_h\nonumber\\
	&+2t\overline\partial(Gb_h\partial\overline\partial a_h)
	+2\partial\overline\partial(Gb_hp\overline\partial a_h),\label{eq: Hessian-H1}\\
	H_0={}&\overline\partial^2(Gb_h\partial^2a_h)
	+\partial^2(Gb_h\overline\partial^2a_h)
	+2\partial\overline\partial(Gb_h\partial\overline\partial a_h).\label{eq: Hessian-H0}
\end{align}
Only two derivatives of $G$ occur; their first weak derivatives are bounded because $G\in W^{3,\infty}$. At the critical point, $H_2(\xi)=\lambda^2\overline\partial^2(Gm_h)(\xi)$.

For $S_1$, take $a_h=a_+$ and let $b_h$ be the finite adjoint amplitude. Its constant coefficient is $a_+^{-1}$, so $m_0=1$. Equations \eqref{eq: leading-stationary-limit} and \eqref{eq: asymptotic-tail-bound} give
\begin{equation}\label{limit for S_1}
	\lim_{h\to0}hS_1=8\pi\lambda^2a_-(\xi)\overline\partial^2G(\xi).
\end{equation}
The $H_1$ integral tends to zero after multiplication by $h$, by oscillatory convergence; the $H_0$ integral has an additional factor $h$. Thus the coefficients involving derivatives of $a_+$ have not been omitted.

We next separate the derivatives of the exponential factors from those of the amplitudes. Define
\begin{equation}\label{def: tilde G_0 in conformal}
	\widetilde G_0=a_+a_-G
\end{equation}
and $S_{1,1}=4\int\widetilde G_0\mathbf v^*(\partial^2e^{\Phi_1/h})(\overline\partial^2e^{\overline{\Phi_2}/h})\,dA$, with $S_{1,2}=S_1-S_{1,1}$. Split the adjoint amplitude into its leading part and correction:
\begin{align}\label{eq:maintermS11mr}
	S_{1,1}^{m}=4\int Ga_-e^{\Phi^*/h}(\partial^2e^{\Phi_1/h})(\overline\partial^2e^{\overline{\Phi_2}/h})\,dA,\quad 	S_{1,1}^{r}=S_{1,1}-S_{1,1}^{m}.
\end{align}
Twice integrating $\overline\partial$ by parts gives the exact formula
\begin{equation}\label{comp S_1,1 m-1}
	S_{1,1}^{m}=4\int E_h\overline\partial^2(Ga_-)(h^{-2}p^2+h^{-1}p')\,dA.
\end{equation}
Therefore,
\begin{equation}\label{stationary phase of S_1,1^m}
	\frac{h}{2\pi}S_{1,1}^{m}\longrightarrow4\lambda^2\overline\partial^2(Ga_-)(\xi),
\end{equation}
and
\begin{equation}\label{limit for S_1^m}
	\lim_{h\to0}hS_{1,1}^{m}=8\pi\lambda^2\overline\partial^2(Ga_-)(\xi).
\end{equation}
For $S_{1,1}^{r}$, expand $r_*$ to sufficiently high order before integrating by parts. Every coefficient of the resulting polynomial begins with $h$, so \eqref{eq: leading-stationary-limit} gives zero. The discarded tail satisfies \eqref{eq: asymptotic-tail-bound}. Thus
\begin{equation}\label{limit for S_1,1}
	\lim_{h\to0}hS_{1,1}=8\pi\lambda^2\overline\partial^2(Ga_-)(\xi).
\end{equation}

To identify the amplitude contribution, set
\begin{equation}\label{def: tilde G_1 in conformal}
	\widetilde G_1=G\overline\partial a_-.
\end{equation}
In $4\partial^2v_{10}\overline\partial^2v_{20}$, the term containing $2\overline\partial a_-\,\overline\partial e^{\overline{\Phi_2}/h}$ gives
\begin{equation}\label{comp S_2,1-1}
	8\int Gb_ha_+\overline\partial a_-\,e^{\Phi^*/h}(\partial^2e^{\Phi_1/h})(\overline\partial e^{\overline{\Phi_2}/h})\,dA
	=-8\int E_h\overline\partial(Gb_ha_+\overline\partial a_-)(h^{-2}p^2+h^{-1}p')\,dA.
\end{equation}
Its limit after multiplication by $h$ is $-16\pi\lambda^2\overline\partial(G\overline\partial a_-)(\xi)$. In addition, the term with $\overline\partial^2a_-$ contributes $8\pi\lambda^2G(\xi)\overline\partial^2a_-(\xi)$. Terms containing $2p\partial a_h/h+\partial^2a_h$ in the first product have at most $h^{-1}$ after the two antiholomorphic integrations by parts. In the second and third products of \eqref{eq: full-Hessian-contraction}, the possible $h^{-2}$ coefficients contain $t^2$ and $t$, respectively, as displayed in \eqref{eq: Hessian-H2}. They vanish at $\xi$. Hence
\begin{equation}\label{limit of S_1,2}
	\lim_{h\to0}hS_{1,2}
	=-16\pi\lambda^2\overline\partial(G\overline\partial a_-)(\xi)
	+8\pi\lambda^2G(\xi)\overline\partial^2a_-(\xi).
\end{equation}
Adding \eqref{limit for S_1,1} and \eqref{limit of S_1,2} gives \eqref{limit for S_1}.

\subsubsection{\it Step 2-2. Dominating term II from the local $S_5$-contribution.}\label{sec:Step22}

For the first derivative terms, use $C^k\partial_k=C_+\partial+C_-\overline\partial$. Formula \eqref{Hessian matrix in complex variable} gives the complete decomposition
\begin{equation}\label{S_5-1 decomp}
	\begin{aligned}
		\tr(D^2u\,C^k\partial_kv)
		=(4e\partial^2u+4\sigma\overline\partial^2u-2b_+\partial\overline\partial u)\partial v+(4d\partial^2u+4\tau\overline\partial^2u-2b_-\partial\overline\partial u)\overline\partial v.
	\end{aligned}
\end{equation}
This includes the trace part of $C^k$, because $\tr C_\pm=-b_\pm$.

For finite amplitudes $a_h,b_h$, write the coefficients of $\partial v$ and $\overline\partial v$, after multiplication by $Gb_he^{\Phi^*/h}$ and removal of $e^{s/h}$, as $A_2/h^2+A_1/h+A_0$ and $B_2/h^2+B_1/h+B_0$. Direct expansion gives
\begin{equation}\label{eq: S5-coefficients}
	\begin{aligned}
		A_2&=4eGm_hp^2,&B_2&=4dGm_hp^2,\\
		A_1&=Gb_h\big(4e(2p\partial a_h+p'a_h)-2b_+p\overline\partial a_h\big),&
		B_1&=Gb_h\big(4d(2p\partial a_h+p'a_h)-2b_-p\overline\partial a_h\big),\\
		A_0&=Gb_h(4e\partial^2a_h+4\sigma\overline\partial^2a_h-2b_+\partial\overline\partial a_h),&
		B_0&=Gb_h(4d\partial^2a_h+4\tau\overline\partial^2a_h-2b_-\partial\overline\partial a_h).
	\end{aligned}
\end{equation}
Integrating once against the derivatives of $v=a_-e^{-\bar s/h}\zeta_h$ gives
\[
S_5[a_h,b_h;\zeta_h]=\int_D E_ha_-\zeta_h(h^{-3}J_3^{(5)}+h^{-2}J_2^{(5)}+h^{-1}J_1^{(5)}+J_0^{(5)})\,dA,
\]
where
\begin{equation}\label{eq: S5-transferred-coefficients}
	\begin{aligned}
		J_3^{(5)}&=-tA_2,&
		J_2^{(5)}&=-\partial A_2-\overline\partial B_2-tA_1,\\
		J_1^{(5)}&=-\partial A_1-\overline\partial B_1-tA_0,&
		J_0^{(5)}&=-\partial A_0-\overline\partial B_0.
	\end{aligned}
\end{equation}
All terms from \eqref{S_5-1 decomp} are present in these four coefficients.

Set $f_5=-A_2$, so $J_3^{(5)}=tf_5$. Since $\partial E_h=tE_h/h$, the cubic term is integrated once more. Its local quadratic coefficient becomes
\begin{equation}\label{eq: S5-local-coefficient}
	K_5=a_-J_2^{(5)}-\partial(a_-f_5)
	=(\partial a_-)A_2-a_-\overline\partial B_2-a_-tA_1.
\end{equation}
For $\zeta_h=1$, there is no derivative of a critical correction. At $\xi$,
\[
K_5(\xi)=4\lambda^2a_-\big((\partial\log a_-)eGm_h-\overline\partial(dGm_h)\big)(\xi).
\]
For $S_{5,1}$, the constant coefficient of $m_h$ is one. Thus
\begin{equation}\label{limit S_5-1}
	\lim_{h\to0}hS_{5,1}
	=8\pi\lambda^2a_-\big((b_-/4)eG-\overline\partial(dG)\big)(\xi).
\end{equation}

We also calculate the split used for the exponent derivatives alone. With $\widetilde\mu=\mu^{-1}a_+a_-$,
\begin{equation}\label{tilde mu}
	\widetilde\mu=\mu^{-1}F_{A_1}^{-1}F_{\overline{A_2}},
\end{equation}
the corresponding part is
\begin{equation}\label{S_5,1^m}
	S_{5,1}^{m}=4\int_D Gm_ha_-d\,e^{s/h}(\partial^2e^{\Phi_1/h})e^{-\Phi_1/h}\overline\partial(e^{-\bar s/h})\,dA.
\end{equation}
Since $\overline\partial E_h=-(\bar w/2h)E_h$, integration by parts gives
\begin{equation}\label{S_5,1^m comp}
	S_{5,1}^{m}=-4\int_D E_h\overline\partial(Gm_ha_-d)(h^{-2}p^2+h^{-1}p')\,dA.
\end{equation}
Consequently, $\lim hS_{5,1}^{m}=-8\pi\lambda^2\overline\partial(Ga_-d)(\xi)$. The difference $S_{5,1}^{r}=S_{5,1}-S_{5,1}^{m}$ has limit
\[
\lim_{h\to0}hS_{5,1}^{r}
=8\pi\lambda^2G\big(e\partial a_-+d\overline\partial a_-\big)(\xi).
\]
These are precisely the leading terms in which the derivative of the second solution falls on $a_-$. Thus $S_{5,1}^{r}$ is not generally $O(1)$.

For the local $S_{6,1}$, all coefficients can likewise be written out. The multipliers of $\partial^2v$, $\overline\partial^2v$ and $\partial\overline\partial v$ are, respectively, $A_1/h+A_0$, $B_1/h+B_0$, and $D_1/h+D_0$, where now
\begin{equation}\label{eq: S6-coefficients}
	\begin{aligned}
		A_1&=4eGm_hp,&A_0&=4Gb_h(e\partial a_h+d\overline\partial a_h),\\
		B_1&=4\sigma Gm_hp,&B_0&=4Gb_h(\sigma\partial a_h+\tau\overline\partial a_h),\\
		D_1&=-2b_+Gm_hp,&D_0&=-2Gb_h(b_+\partial a_h+b_-\overline\partial a_h).
	\end{aligned}
\end{equation}
Two integrations by parts give
\begin{equation}\label{eq: S6-transferred-coefficients}
	\begin{aligned}
		J_3^{(6)}&=t^2A_1,\\
		J_2^{(6)}&=t^2A_0+2t\partial A_1+t'A_1+t\overline\partial D_1,\\
		J_1^{(6)}&=\partial^2A_1+2t\partial A_0+t'A_0+\overline\partial^2B_1+t\overline\partial D_0+\partial\overline\partial D_1,\\
		J_0^{(6)}&=\partial^2A_0+\overline\partial^2B_0+\partial\overline\partial D_0.
	\end{aligned}
\end{equation}
Its integral has the same powers $h^{-3},h^{-2},h^{-1},1$. Put $f_6=tA_1$. After the additional integration using $\partial E_h=tE_h/h$, the local coefficient is $K_6=a_-J_2^{(6)}-\partial(a_-f_6)$. At $\xi$, $J_2^{(6)}=t'A_1$ and $\partial(a_-f_6)=a_-t'A_1$, so $K_6(\xi)=0$.

For the exponent-only term, this cancellation is also seen directly:
\begin{equation}\label{S_6-1^m comp}
	S_{6,1}^{m}
	=4\int_D Gm_ha_-\sigma\,\frac p h\,e^{s/h}\overline\partial^2(e^{-\bar s/h})\,dA
	=\frac4h\int_D E_hp\,\overline\partial^2(Gm_ha_-\sigma)\,dA.
\end{equation}
It follows that $hS_{6,1}^{m}\to0$. Equations \eqref{eq: S6-coefficients}--\eqref{eq: S6-transferred-coefficients} include all amplitude derivatives, and the calculation of $K_6(\xi)$ gives
\begin{equation}\label{limit S_6-1}
	\lim_{h\to0}hS_{6,1}=0.
\end{equation}

\subsubsection{Step 2-3. Nonlocal terms from the critical correction.}\label{sec:Step23}

We now take $\zeta_h=\widetilde r_2$ in the transferred formulas. The following calculation applies to both mixed groups. If the cubic coefficient is $tf_h$, then
\begin{equation}\label{eq: cubic-critical-transfer}
	\begin{aligned}
		h^{-3}\int_D E_ha_-tf_h\widetilde r_2\,dA
		=-h^{-2}\int_D E_h\partial(a_-f_h)\widetilde r_2\,dA-h^{-2}\int_D E_ha_-f_h\partial\widetilde r_2\,dA.
	\end{aligned}
\end{equation}
The last term is evaluated, not estimated away. By \eqref{eq: representation of r_1 and r_2},
\begin{equation}\label{eq:tilder2}
	\partial\widetilde r_2=E_h^{-1}W_1\widetilde s_2.
\end{equation}
Since $a_-W_1=a_+$ on the support of $f_h$, the last term in \eqref{eq: cubic-critical-transfer} equals
\begin{equation}\label{troublesome term1}
	\begin{aligned}
		-h^{-2}\int_D a_+f_h\widetilde s_2\,dA
		=-h^{-2}\int_D a_+f_h\mathscr C(E_hW_0(1+\widetilde r_2))\,dA=h^{-2}\int_D E_hW_0\mathscr C(a_+f_h)(1+\widetilde r_2)\,dA.
	\end{aligned}
\end{equation}
The last equality uses \eqref{eq:integration by parts for d}. The last integral is over the entire support of $W_0$ in $D$, not only the original domain.

The remaining local quadratic coefficient is $K_h=a_-J_2-\partial(a_-f_h)$. Its contribution after multiplication by $h$ is $h^{-1}\int E_hK_h\widetilde r_2$, which tends to zero by \eqref{eq: pairing-limit-rule}. The Cauchy product $W_0\mathscr C(a_+f_h)$ is compactly supported in $D$ and uniformly bounded in $W^{1,p_0}$. It is smooth near $\xi$ by local elliptic regularity for $\overline\partial\mathscr C(a_+f_h)=a_+f_h$, since $f_h$ is smooth there. Its part containing $\widetilde r_2$ satisfies the same estimate. Its part without $\widetilde r_2$ has limit $2\pi W_0(\xi)\mathscr C(a_+f_0)(\xi)$ by \eqref{eq: leading-stationary-limit}.

The terms of order $h^{-1}$ and $1$ in the original transferred formula give, after multiplication by $h$, integrals bounded by $C\|\widetilde r_2\|_2$ and $Ch\|\widetilde r_2\|_2$. They vanish by \eqref{eq:rh_L2_estim}. Thus every term of the critical calculation has been treated.

For $S_{5,3}$, $f_5=-4eGm_hp^2$, so
\begin{equation}\label{eq: S53-limit}
	\lim_{h\to0}hS_{5,3}=-8\pi W_0(\xi)\mathscr C(a_+eGp^2)(\xi).
\end{equation}
We also record separately the antiholomorphic exponent-derivative terms coming from $C^1$ and $C^2$. In the finite-amplitude calculation for $S_{5,3}$, take $a_h=a_+$ and $m_h=b_ha_+$. Set $k_{1,h}=Ga_-m_h\tr(AC^1)$ and $k_{2,h}=-\mathsf{i}Ga_-m_h\tr(AC^2)$, and define, for $j=1,2$,
\[
S_{5,3}^{1,j}=\int_U k_{j,h}e^{s/h}(h^{-2}p^2+h^{-1}p')\overline\partial(e^{-\bar s/h}\widetilde r_2)\,dA.
\]
The factors $1$ and $-\mathsf{i}$ are the coefficients of $\overline\partial$ in $\partial_{x_1}$ and $\partial_{x_2}$. Terms in which the derivative falls on $a_-$ are already included in \eqref{eq: S5-coefficients}. Each $k_{j,h}$ is compactly supported and belongs to $W^{3,\infty}$, with uniformly bounded norms for the finite expansion. For $k_h=k_{j,h}$, integration by parts gives
\begin{equation}\label{S_5,3,1,1-1}
	\int k_he^{s/h}(h^{-2}p^2+h^{-1}p')\overline\partial(e^{-\bar s/h}\widetilde r_2)\,dA
	=-\int E_h\overline\partial k_h(h^{-2}p^2+h^{-1}p')\widetilde r_2\,dA.
\end{equation}
After finite expansion of the noncritical amplitude, the first term is covered by \eqref{eq: pairing-limit-rule} and the second by the $L^2$ bound. Therefore
\begin{equation}\label{limit S_5,3,1}
	hS_{5,3}^{1,1}\to0,\quad hS_{5,3}^{1,2}\to0.
\end{equation}
The holomorphic derivative of the critical correction produces \eqref{eq: S53-limit}.

For the finite-amplitude expression of $S_{6,3}$, define the three scalar parts by
\begin{equation}\label{S_6,3}
	\begin{aligned}
		S_{6,3}^{1}&=\int_U e^{s/h}(B_1/h+B_0)\overline\partial^2v_{21}\,dA,\\
		S_{6,3}^{2}&=\int_U e^{s/h}(A_1/h+A_0)\partial^2v_{21}\,dA,\\
		S_{6,3}^{3}&=\int_U e^{s/h}(D_1/h+D_0)\partial\overline\partial v_{21}\,dA.
	\end{aligned}
\end{equation}
Here $v_{21}=a_-e^{-\bar s/h}\widetilde r_2$, and the coefficients are exactly those in \eqref{eq: S6-coefficients}. The finite-amplitude replacement satisfies $h(S_{6,3}-S_{6,3}^{1}-S_{6,3}^{2}-S_{6,3}^{3})\to0$ by \eqref{eq: asymptotic-tail-bound}. The antiholomorphic second derivative part has at most $h^{-1}$ after two antiholomorphic integrations by parts, so
\begin{equation}\label{limit of S_6,3,1}
	hS_{6,3}^{1}\to0.
\end{equation}
For the mixed derivative part, transfer the antiholomorphic derivative first. The remaining derivative of $e^{-\bar s/h}\widetilde r_2$ is holomorphic and acts only on its amplitude and correction. Its normalized $L^2$ size is $O(h^{1/2+\epsilon})$, while the first-solution coefficient has at most $h^{-1}$. Hence its contribution after multiplication by $h$ tends to zero.

The holomorphic second derivative part is different. The cubic coefficient in \eqref{eq: S6-transferred-coefficients} gives $f_6=4teGm_hp$, and \eqref{troublesome term1} yields
\begin{equation}\label{S_6,3 comp 1}
	\lim_{h\to0}hS_{6,3}=8\pi W_0(\xi)\mathscr C(a_+eGtp)(\xi).
\end{equation}
This is generally nonzero. Combining \eqref{eq: S53-limit} and \eqref{S_6,3 comp 1}, and using $tp-p^2=pq$, gives
\begin{equation}\label{limit of S_3,3,2+S_6,3,2}
	\lim_{h\to0}h(S_{5,3}+S_{6,3})
	=8\pi W_0(\xi)\mathscr C(a_+eGpq)(\xi).
\end{equation}
The mixed derivative part just considered satisfies
\begin{equation}\label{limit S_6,3,3}
	hS_{6,3}^{3}\to0.
\end{equation}
Hence, the holomorphic second derivative part gives the nonlocal limit in \eqref{S_6,3 comp 1}, while the other two parts have zero scaled limit.

\subsection{Step 3. Remaining terms: lower order contributions}\label{sec:Step3}

We now verify the remaining groups using their complete coefficients, rather than the number of explicit powers of $h$ alone.

\paragraph{\it Step 3-1. The term $S_2$.}
The identities $\tr(BD^2f)=4\partial^2f$ and \eqref{eq: full-Hessian-contraction} give
\begin{equation}\label{S_2 comp 1}
	\begin{aligned}
		S_2=4\int G\mathbf v^*\big[
		\partial^2v_{11}\overline\partial^2v_{20}
		+\overline\partial^2v_{11}\partial^2v_{20}
		+2\partial\overline\partial v_{11}\partial\overline\partial v_{20}
		\big]\,dA.
	\end{aligned}
\end{equation}
Use the finite amplitude for $a_+r_1$ in \eqref{eq: Hessian-H2}--\eqref{eq: Hessian-H0}. This amplitude and every fixed spatial derivative of its polynomial begin with $h$. Thus the constant coefficient of $H_2$ is zero. Formula \eqref{eq: leading-stationary-limit} gives zero for the leading part after multiplication by $h$. The $H_1$ and $H_0$ terms vanish by oscillatory convergence and their additional factors. The tail is covered by \eqref{eq: asymptotic-tail-bound}. Hence, 
\begin{equation}\label{limit for S_2}
	\lim_{h\to0}hS_2=0.
\end{equation}

\paragraph{\it Step 3-2. The term $S_3$.}
We also have $\tr(AD^2f)=4\overline\partial^2f$. The exact decomposition is
\begin{equation}\label{S_3 comp1}
	S_3=4\int G\mathbf v^*\big[
	\partial^2v_{10}\overline\partial^2v_{21}
	+\overline\partial^2v_{10}\partial^2v_{21}
	+2\partial\overline\partial v_{10}\partial\overline\partial v_{21}
	\big]\,dA.
\end{equation}
After the transfers in \eqref{eq: Hessian-transferred}, this equals $4\int E_ha_-\widetilde r_2(h^{-2}H_2+h^{-1}H_1+H_0)\,dA$, up to the quantified adjoint tail. Each coefficient has the required $W^{1,p_0}$ bound. Therefore
\begin{equation}\label{comp of S_3,1}
	\bigg|h^{-1}\int E_ha_-H_2\widetilde r_2\,dA\bigg|\leq Ch^{2\epsilon},
\end{equation}
while the remaining two terms after multiplication by $h$ are bounded by $C\|\widetilde r_2\|_2$ and $Ch\|\widetilde r_2\|_2$. This treats the terms containing derivatives of both amplitudes as well. We obtain
\begin{equation}\label{limit for S_3}
	\lim_{h\to0}hS_3=0.
\end{equation}

\paragraph{\it Step 3-3. The term $S_4$.}
Using both corrections, the same matrix multiplication gives
\begin{equation}\label{S_4 term by term}
	\tr(D^2v_{11}D^2v_{21})
	=4\overline\partial^2v_{11}\partial^2v_{21}
	+4\partial^2v_{11}\overline\partial^2v_{21}
	+8\partial\overline\partial v_{11}\partial\overline\partial v_{21}.
\end{equation}
Let $S_{4,1},S_{4,2},S_{4,3}$ be the integrals of these three terms multiplied by $G\mathbf v^*$, so
\begin{equation}\label{S_4 comp 1}
	S_4=S_{4,1}+S_{4,2}+S_{4,3}.
\end{equation}
Transfer the derivatives of $v_{21}$ as in \eqref{eq: Hessian-transferred}, now taking $a_h$ to be the finite amplitude of $a_+r_1$. The coefficients in each of the three resulting integrals are bounded in $W^{1,p_0}$ and begin with $h$. The worst possible integral after multiplication by $h$ still has the form $h^{-1}\int E_hf_h\widetilde r_2$, which tends to zero by \eqref{eq: pairing-limit-rule}. Terms with fewer inverse powers of $h$ are bounded directly by $\|\widetilde r_2\|_2$. The tails satisfy \eqref{eq: asymptotic-tail-bound}. Thus all three subterms vanish in the scaled limit, and
\begin{equation}\label{limit for S_4}
	\lim_{h\to0}hS_4=0.
\end{equation}

\paragraph{\it Step 3-4. The terms $S_{5,2}$ and $S_{6,2}$.}
For these terms, $\zeta_h=1$ and the finite first amplitude has zero constant term. In \eqref{eq: S5-coefficients} and \eqref{eq: S6-coefficients}, all coefficients therefore begin with $h$. The cubic terms are first integrated using $\partial E_h=tE_h/h$; this produces the quadratic local coefficients $K_5,K_6$, which also have zero constant term. Formula \eqref{eq: leading-stationary-limit} gives zero for their scaled limits. The lower coefficients vanish by oscillatory convergence, and the tails satisfy \eqref{eq: asymptotic-tail-bound}. Consequently,
\begin{equation}\label{limit S52 and S62}
	\lim_{h\to0}hS_{5,2}=\lim_{h\to0}hS_{6,2}=0.
\end{equation}

\paragraph{\it Step 3-5. The terms $S_{5,4}$ and $S_{6,4}$.}
Here the first amplitude again has zero constant term and $\zeta_h=\widetilde r_2$. Apply the exact transfer \eqref{eq: cubic-critical-transfer}--\eqref{troublesome term1}. The local quadratic term containing $\widetilde r_2$ tends to zero by \eqref{eq: pairing-limit-rule}. The Cauchy contribution without $\widetilde r_2$ has zero constant coefficient because $f_h$ begins with $h$. Its contribution with $\widetilde r_2$ tends to zero by the same pairing estimate. The lower terms are bounded by $\|\widetilde r_2\|_2$ after multiplication by $h$. Thus,
\begin{align}\label{limit of S_5,6,4}
	\lim_{h\to0}hS_{5,4}=\lim_{h\to0}hS_{6,4}=0.
\end{align}
This argument uses the differentiated correction equation and does not infer these limits from an unweighted estimate of a growing exponential.

\paragraph{\it Step 3-6. The term $S_7$.}
Set $Z_{\alpha\beta}=\tr(C_\alpha C_\beta)$ for $\alpha,\beta\in\{+,-\}$. The coefficients of $\partial v$ and $\overline\partial v$ after removing $e^{s/h}$ are $A_1/h+A_0$ and $B_1/h+B_0$, where
\[
A_1=Gm_hpZ_{++},\quad A_0=Gb_h(Z_{++}\partial a_h+Z_{-+}\overline\partial a_h),
\]
and $B_1=Gm_hpZ_{+-}$, $B_0=Gb_h(Z_{+-}\partial a_h+Z_{--}\overline\partial a_h)$. Transferring the derivative on the second solution gives
\[
S_7=\int E_ha_-(1+\widetilde r_2)(h^{-2}J_2+h^{-1}J_1+J_0)\,dA,
\]
up to the noncritical tails, with $J_2=-tA_1$, $J_1=-\partial A_1-\overline\partial B_1-tA_0$, and $J_0=-\partial A_0-\overline\partial B_0$. At $\xi$, $J_2=0$. The leading integral without a correction therefore has zero limit by \eqref{eq: leading-stationary-limit}; the part with $\widetilde r_2$ vanishes by \eqref{eq: pairing-limit-rule}. The lower terms and tails have already established bounds. Hence,
\begin{equation}\label{limit of S_7}
	\lim_{h\to0}hS_7=0.
\end{equation}

\section{Unique determinations of the conformal factor and the source}\label{sec: det of conformal factor}

\subsection{The proof of Theorem \ref{theorem: uniqueness of conformal factor}}

\begin{proof}[Proof of Theorem \ref{theorem: uniqueness of conformal factor}]
	For each $\xi\in U\setminus\partial\widetilde\Omega$, add the limits in \eqref{limit for S_1}, \eqref{limit S_5-1}, \eqref{limit S_6-1} and \eqref{limit of S_3,3,2+S_6,3,2}. The other groups vanish by Section~\ref{sec:Step3}. Since $I_h=0$,
	\begin{equation}\label{governing term in asym analy 1}
		\begin{aligned}
			0=8\pi\big[\lambda^2a_-\big(\overline\partial^2G-\overline\partial(dG)+(b_-/4)eG\big)+W_0\mathscr C\big(a_+eG(-\lambda^2+(z-\xi)^2/16)\big)\big](\xi).
		\end{aligned}
	\end{equation}
	Here $e=\overline\tau$ and $d=\overline\sigma$. The nonlocal term is the sum of the two separate critical contributions; both have been included.
	
	Put $f=a_+eG$, $M_0=\pi^{-1}\int_\C f(w)\,dA(w)$ and $M_1=\pi^{-1}\int_\C wf(w)\,dA(w)$. Define the holomorphic polynomial $\Pi_G(z)=(zM_0-M_1)/16$. The shift of the critical point gives the exact kernel identity
	\begin{equation}\label{governing term in asym analy 2}
		\begin{aligned}
			\mathscr C\big((-\lambda^2+(w-\xi)^2/16)f(w)\big)(\xi)=-\lambda^2\mathscr Cf(\xi)+\frac1{16\pi}\int_\C(\xi-w)f(w)\,dA(w)=-\lambda^2\mathscr Cf(\xi)+\Pi_G(\xi).
		\end{aligned}
	\end{equation}
	The equality uses $(w-\xi)^2/(\xi-w)=\xi-w$. The factor $W_0(\xi)$ outside the Cauchy operator in \eqref{governing term in asym analy 1} multiplies both terms on the right.
	
	Divide \eqref{governing term in asym analy 1} by $8\pi\lambda^2a_-(\xi)$. Since $W_0/a_-=q_-/(4a_+)$ on $U$, we obtain
	\begin{equation}\label{governing term in asym analy 3}
		\overline\partial(\overline\partial G-\overline\sigma G)+\frac{b_-}{4}\overline\tau G
		=\frac{q_-}{4a_+}\big(\mathscr C(a_+\overline\tau G)-\lambda^{-2}\Pi_G\big).
	\end{equation}
	It first holds away from $\partial\widetilde\Omega$. Both sides are continuous on $U$, since $G\in C^2(U)$ and its Cauchy integral is continuous. Thus it holds throughout $U$.
	
	Apply Lemma~\ref{lemma: UCP with Carleman} with $\rho=G$, $A_0=1$, $\alpha=-\overline\sigma$, $\delta=(b_-/4)\overline\tau$, $\beta=q_-/(4a_+)$, $\gamma=a_+\overline\tau$, and $P=-\lambda^{-2}\Pi_G$. All coefficients are smooth and $P$ is holomorphic. The extension in Remark~\ref{remark: large U} gives $G=0$ on the exterior disk $B_0$. On this disk $X=(\kappa,0)$, so $q_-=-\overline\partial b_-+b_+b_-/4=\kappa^2/4$. Since $a_+$ never vanishes, $\beta=\kappa^2/(16a_+)$ is nonzero there. The UCP therefore gives $G=0$ throughout $U$. By $G=\mu^{-1}(1-c)$, we conclude that $c=1$ in $\widetilde\Omega$.
\end{proof}

\subsection{Proof of Theorem \ref{theorem: uniqueness}}\label{sec: proof of uniqueness}
					
\begin{proof}[Proof of Theorem \ref{theorem: uniqueness}]
	Let $u_0^{(j)}$ solve \eqref{MA equation zero boundary j=1,2} and put $g_j=D^2u_0^{(j)}$.
	\begin{itemize}
		\item \textit{Step 1.} The source jets through order two agree by hypothesis, and \eqref{eq: same DN map} gives $\Lambda_{F_1}(0)=\Lambda_{F_2}(0)$. Lemma~\ref{lemma: boundary determination} therefore gives equality of the boundary derivatives of $u_0^{(1)},u_0^{(2)}$ through order four. In particular, the jets of $g_1,g_2$ through order two agree.
		
		\item \textit{Step 2.} Smoothness of the solution maps permits differentiation of \eqref{eq: same DN map}. Therefore, the first and second derivatives of the nonlinear DN maps agree in all smooth real directions, as do their complex multilinear extensions. Lemma~\ref{lemma: nonlinear to linear DN} gives
		\begin{equation}\label{linearized DN same in the proof}
			\Lambda'_{g_1}\phi=\Lambda'_{g_2}\phi\quad\text{for every }\phi\in C^\infty(\partial\Omega).
		\end{equation}
		
		\item \textit{Step 3.} The domain is simply connected because it is uniformly convex. The boundary hypotheses of Theorem~\ref{theorem: J=Id} hold by Step 1. That theorem gives $g_1=cg_2$ in the original coordinates, with $c>0$ and $c|_{\partial\Omega}=1$. Differentiating this equality with the known metric jets shows that the boundary derivatives of $c-1$ through order two vanish.
		
		\item \textit{Step 4.} Equality of the second derivatives of the nonlinear DN maps gives \eqref{integral id second 2}. The separate metric and volume transformations in Section~\ref{sec: the second integral identity} give \eqref{integral id.}, with $G=\mu^{-1}(1-c)$. The CGO calculations of Section~\ref{sec: asymptotic analysis} and the Carleman continuation argument prove Theorem~\ref{theorem: uniqueness of conformal factor}, so $c=1$. Hence $D^2u_0^{(1)}=D^2u_0^{(2)}$. Taking determinants gives $F_1=F_2$ in $\Omega$.
	\end{itemize}
	This proves the theorem.
\end{proof}

\begin{remark}\label{remark: Hessian compatibility}
	For the Monge--Amp\`ere application, the first three steps of the proof also give a direct conclusion specific to Hessian metrics. They yield $D^2u_0^{(1)}=cD^2u_0^{(2)}$ in the original Euclidean coordinates, with $c>0$ and $c|_{\partial\Omega}=1$. Write $h_{ab}=\partial_{ab}u_0^{(2)}$ and let $(h^{ab})=(h_{ab})^{-1}$. Differentiating the equality of Hessians and interchanging the indices $k,a$ gives
	\[
	(\partial_kc)h_{ab}+c\partial_kh_{ab}=(\partial_ac)h_{kb}+c\partial_ah_{kb}.
	\]
	The symmetry of third derivatives cancels the last terms. For fixed $k$, multiplication by $h^{ab}$ and summation over $a,b$ therefore give
	\[
	2\partial_kc=\sum_{a,b=1}^2(\partial_kc)h_{ab}h^{ab}
	=\sum_{a=1}^2(\partial_ac)\sum_{b=1}^2h_{kb}h^{ab}
	=\partial_kc.
	\]
	Thus $\nabla c=0$. Connectedness and the boundary value imply $c=1$, and taking determinants gives $F_1=F_2$. This argument uses the Hessian compatibility in the original coordinates. It is separate from the second linearized integral argument used in the main proof.
\end{remark}

\section*{Statements and Declarations}
					
\para{Data availability statement} No datasets were generated or analyzed during the current study.

\para{Conflict of interest} The authors declare that they have no conflict of interest.

\para{Acknowledgments} The authors would like to thank Professor Mikko Salo for fruitful discussions to improve this article. T.~L. was partly supported by the Academy of Finland (Centre of Excellence in Inverse Modelling and Imaging and FAME Flagship, grant numbers 312121 and 359208). Y.-H. L. is partially supported by the National Science and Technology Council (NSTC) of Taiwan, under project 113-2115-M-A49-017-MY3. Y.-H. L. acknowledges financial support from the Alexander von Humboldt Foundation through the Henriette Herz Scouting Programme during his research stay at the University of Duisburg-Essen.

					\bibliography{ref} 
					
					\bibliographystyle{alpha}
					
				\end{document}